\documentclass[12pt]{article}
\usepackage{color}
\usepackage{amsmath}
\usepackage{physics}
\usepackage{amsmath}
\usepackage{tikz}
\usepackage{mathdots}
\usepackage{yhmath}
\usepackage{cancel}
\usepackage{color}
\usepackage{siunitx}
\usepackage{array}
\usepackage{multirow}
\usepackage{amssymb}
\usepackage{gensymb}
\usepackage{tabularx}
\usepackage{extarrows}
\usepackage{booktabs}
\usetikzlibrary{fadings}
\usetikzlibrary{patterns}
\usetikzlibrary{shadows.blur}
\usetikzlibrary{shapes}

\usepackage{extarrows}
\usepackage{geometry}
\usepackage{authblk} 
\usepackage{hyperref} 
\usepackage{color}
\usepackage{ntheorem}
\usepackage{enumerate}

\usepackage[mathscr]{euscript}
\usepackage[utf8]{inputenc}
\def\0{\emptyset}

\newtheorem{theorem}{Theorem}[section]

\newtheorem{lemma}[theorem]{Lemma}
\newtheorem{claim}[theorem]{Claim}

\newtheorem{cor}[theorem]{Corollary}
\newtheorem{proposition}[theorem]{Proposition}

\newtheorem{construction}[theorem]{Construction}
\newtheorem{conjecture}[theorem]{Conjecture}
\newenvironment{proof}{{\noindent\it Proof.}}{\hfill $\square$\par}

\usepackage{graphicx}

\usepackage{enumerate}
\usepackage{enumitem}

\usepackage{
	amsmath,			
	amssymb,			
	enumerate,		    
	graphicx,			
	lastpage,			
	multicol,			
	multirow,			
	pifont,			    
}

\usepackage[numbers]{natbib}

\newcommand{\lc}{\left\lceil}
\newcommand{\rc}{\right\rceil}
\newcommand{\cA}{\mathcal{A}}
\newcommand{\cB}{\mathcal{B}}

\newcommand{\cD}{\mathcal{D}}
\newcommand{\cE}{\mathcal{E}}
\newcommand{\cF}{\mathcal{F}}
\newcommand{\cG}{\mathcal{G}}

\newcommand{\cH}{\mathcal{H}}

\newcommand{\cK}{\mathcal{K}}

\newcommand{\cL}{\mathcal{L}}
\newcommand{\cM}{\mathcal{M}}

\newcommand{\cS}{\mathcal{S}}
\newcommand{\cT}{\mathcal{T}}

\newcommand\ex{\ensuremath{\mathrm{ex}}}
\newcommand{\exsumr}{\text{ex}^{\sum}_r}
\newcounter{cases}
\newcounter{subcases}[cases]

\begin{document}


\title{Hypergraph extensions of the Alon--Frankl Theorem and rainbow hyper-Tur\'an problems 
}
\author{
Xiamiao Zhao\thanks{\small Department of Mathematical Sciences, Tsinghua University, Beijing 100084, China. Email:
\small \texttt{zxm23@mails.tsinghua.edu.cn}}\,
\hspace{0.2em}
Yuanpei Wang\thanks{\small  \textit{Corresponding author}. Department of Mathematics, Shanghai University, Shanghai 200444, P.R. China. Email:
\small \texttt{boyuan@shu.edu.cn}}\,
\hspace{0.2em} 
Junpeng Zhou\thanks{\small Department of Mathematics, Shanghai University, Shanghai 200444, P.R. China. Email:
\small \texttt{junpengzhou@shu.edu.cn}.} \thanks{\small Newtouch Center for Mathematics of Shanghai University, Shanghai 200444, P.R. China.}}

\date{}

\maketitle\baselineskip 16.3pt

\begin{abstract}
Given a graph $F$, the $r$-expansion $F^{(r)+}$ of $F$ is the $r$-uniform hypergraph obtained from $F$ by inserting $r-2$ new distinct vertices in each edge of $F$. 
Recently, Alon and Frankl (JCTB, 2024) and Gerbner (JGT, 2023) studied the maximum number of edges in $n$-vertex $F$-free graphs with bounded matching number, respectively. Gerbner, Tompkins and Zhou (EJC, 2025) considered the analogous Tur\'{a}n problems on hypergraphs with bounded matching number. In this paper, we study hypergraph extensions of the Alon--Frankl Theorem. More precisely, we determine the maximum number of hyperedges in an $n$-vertex $r$-uniform hypergraph containing neither a matching $M^r_{s+1}$ nor the expansion $K_{\ell+1}^{(r)+}$ of the clique $K_{\ell+1}$ for all small $s<\frac{\ell^2-1}{2}$ and all sufficiently large $s$, respectively. 
This result partly confirms a conjecture proposed by Gerbner, Tompkins and Zhou (EJC, 2025).

As a key tool, we determine the rainbow hyper-Tur\'{a}n number for expansions of cliques, which is defined as the maximum sum of size of a sequence of hypergraphs $\mathcal{H}_1,\dots,\mathcal{H}_k$ that contains no rainbow copies of expansions of cliques with given size.
It extends the result of Keevash, Saks, Sudakov and Verstra{\"e}te (AAM, 2004), which determined the rainbow Tur\'an number of cliques in the graph case.  
These results shows a correlation between the hyper-Tur\'an problem and the rainbow hyper-Tur\'an number.
\end{abstract}


\vskip.3cm

\section{Introduction}
An \textit{$r$-uniform hypergraph} ($r$-graph for short) $\cH=(V(\cH),E(\cH))$ consists of a vertex set $V(\cH)$ and a hyperedge set $E(\cH)$, where each hyperedge in $E(\cH)$ is an $r$-subset of $V(\cH)$. The size of $E(\cH)$ is denoted by $e(\cH)$.

Let $\mathcal{F}$ be an $r$-graph. An $r$-graph $\cH$ is \textit{$\mathcal{F}$-free} if $\cH$ does not contain $\mathcal{F}$ as a subhypergraph. The \textit{Tur\'{a}n number} of $\mathcal{F}$, denoted by ${\rm{ex}}_r(n,\mathcal{F})$, is the maximum number of hyperedges in an $n$-vertex $\mathcal{F}$-free $r$-graph. When $r=2$, we use $\mathrm{ex}(n,\mathcal{F})$ instead of $\mathrm{ex}_2(n,\mathcal{F})$. A classical result in extremal graph theory is Tur\'{a}n theorem \cite{turan1941egy}, which determines the exact Tur\'{a}n number for the $\ell$-vertex complete graph $K_\ell$. The Erd\H{o}s-Stone-Simonovits theorem \cite{erdos1966limit,erdos1946structure} gives an asymptotics of the Tur\'{a}n number for any $k$-chromatic graph. When $\cF$ is bipartite, the problem of determining $\ex(n,\cF)$ remains an active topic in extremal graph theory. For an extensive overview of the historical development, we refer the reader to the survey by Mubayi and Verstra{\"e}te \cite{mubayi2016a}. In particular, Erd\H{o}s and Gallai \cite{gallai1959maximal} determined the Tur\'{a}n number of $M_{s+1}$, where $M_{s+1}$ denotes a matching of size $s+1$, i.e., the graph consisting of $s+1$ independent edges.  

Let $K_{k+1}$ denote the complete graph on $k+1$ vertices, and let $G(n,\ell,s)$ denote the complete $\ell$-partite graph on $n$ vertices with one part of order $n-s$ and each other part of order $\lfloor\frac{s}{\ell-1}\rfloor$ or $\lceil\frac{s}{\ell-1}\rceil$.
Recently, Alon and Frankl~\cite{alon2024turan} considered Tur\'{a}n problems on graphs with bounded matching number. Specifically, they determined the exact value of ${\rm{ex}}(n,\{M_{s+1},K_{\ell+1}\})$, and showed that for $s\ge s_0(F)$ and $n\ge n_0(F)$, 
\begin{equation} \label{AlonFrankleq}
{\rm{ex}}(n,\{M_{s+1},F\})=|E(G(n,\ell,s))|, 
\end{equation}
where $F$ is an arbitrary color-critical graph of chromatic number $\ell+1$.
In~\cite{gerbner2024turan}, Gerbner generalized~\eqref{AlonFrankleq} as follows.

\begin{theorem}[\cite{gerbner2024turan}]\label{thm1}
If $\chi(F)>2$ and $n$ is sufficiently large, then $\ex(n,\{F,M_{s+1}\})=\ex(s,\cF)+s(n-s)$, where $\cF$ is the family of graphs obtained by deleting an independent set from $F$.
\end{theorem}

In the case $F$ is bipartite, Gerbner~\cite{gerbner2024turan} also determined ${\rm{ex}}(n,\{M_{s+1},F\})$ apart from an additive constant term.

Hypergraph Tur\'{a}n problems are notoriously more difficult than graph versions. A fundamental object of investigation in hypergraph Tur\'{a}n problems is expansions. Given a graph $F$ with vertices $V(F)=\{v_1,\dots,v_{|V(F)|}\}$, the \textit{$r$-uniform expansion} (or briefly $r$-expansion) $F^{(r)+}$ of $F$ is the $r$-graph obtained from $F$ by inserting $r-2$ new distinct vertices in each edge of $F$, such that the $(r-2)e(F)$ new vertices are distinct from each other and are not in $V(F)$. It was introduced by Mubayi in \cite{mubayi2006a}.
And the original vertices $\{v_1,\dots,v_{|V(F)|}\}$ are called the \textit{core vertices} of $F^{(r)+}$.

In 1965, Erd\H{o}s \cite{erdos1965problem} proposed the Erd\H{o}s-Matching Conjecture on the Tur\'an number of an $r$-uniform matching $M^{(r)+}_{s+1}$.

\begin{conjecture}[\cite{erdos1965problem}]\label{con1}
Let integers $r\geq2$, $s\geq1$ and $n\geq (s+1)r-1$. Then $$\ex_r(n,M^{(r)+}_{s+1})\leq \max\left\{ \binom{(s+1)r-1}{r}, \binom{n}{r} - \binom{n-s}{r} \right\}.$$
\end{conjecture}

It is worth knowing that the conjecture is known to hold for the cases $r=2$ \cite{gallai1959maximal} and $r=3$ \cite{frankl2017maximum,luczak2014erdHos}. For the general case, Erd\H{o}s \cite{erdos1965problem} verified Conjecture \ref{con1} for large $n$. Later, the threshold was improved in several papers by Bollob\'as, Daykin, Erd\H{o}s \cite{bollob1976sets}, Huang, Loh, Sudakov \cite{huang2012size}, Frankl, \L uczak, Mieczkowska \cite{luczak2014erdHos}. For further developments and relevant results concerning this conjecture, one can refer to \cite{frankl2013improved,frankl2017proof,frankl2022erdHos,kolupaev2023erdHos}.
Very recently, Kupavskii and Sokolov \cite{kupavskii2025complete} completely resolved the conjecture for the case $n\leq 3(s+1)$.
It is known to hold for sufficiently large $n$, and we state a recent result due to Frankl~\cite{frankl2013improved}.

\begin{theorem}[\cite{frankl2013improved}]\label{emc}
Let $r,s\geq1$ and $n\geq (2s+1)r-s$. Then $\ex_r(n,M^{(r)+}_{s+1})=\sum_{i=1}^s\binom{s}{i}\binom{n-s}{r-i}$. Equality holds only for families isomorphic to $\mathcal{A}(n,r,s)$, where $\mathcal{A}(n,r,s)$ is the $r$-graph consisting of every hyperedge which intersects a fixed $s$-set. 
\end{theorem}

For expansions of complete graphs, Erd\H{o}s \cite{erdHos1971topics} conjectured that ${\rm{ex}}_r(n,K^{(r)+}_3)=\binom{n-1}{r-1}$ for $n\geq \frac{3}{2}r$. The conjecture was later proved by Mubayi and Verstra\"{e}te \cite{mubayi2005a}. Let $\cT_r(n,\ell)$ denote the complete balanced $\ell$-partite $r$-graph. For convenience, let $t_r(n,\ell):=e(\cT_r(n,\ell))$. Mubayi \cite{mubayi2006a} conjectured that $\cT_r(n,\ell)$ is the unique maximum $K^{(r)+}_{\ell+1}$-free $r$-graph for sufficiently large $n$. The conjecture was later proved by Pikhurko \cite{pikhurko2013exact}.

\begin{theorem}[\cite{pikhurko2013exact}]\label{Pikhurko}
Let integers $\ell \geq r\geq 3$. If $n$ is sufficiently large, then $\ex_r(n, K_{\ell+1}^{(r)+}) =t_r(n,\ell)$ and $\cT_r(n,\ell)$ is the unique extremal hypergraph.
\end{theorem}

Recently, Zhou and Yuan \cite{ZhY} considered linear Tur\'{a}n problems on expansions of graphs with bounded matching number. Gerbner, Tompkins and Zhou~\cite{gerbner2025hypergraph} considered the analogous Tur\'{a}n problems for hypergraphs with bounded matching number. In particular, they determined the asymptotics of $\ex_r(n,\{\mathcal{H},M^{(r)+}_{s+1}\})$ where $\cH$ is an arbitrary $r$-graph. Let
\begin{align*}
    \mathcal{A}(\mathcal{H})=\{\mathcal{H}[S]:S\subseteq V(\mathcal{H}(S)),E(\mathcal{H}-S)=\emptyset\}.
\end{align*}
A proper $k$-coloring of $\cH$ is a mapping from $V(\cH)$ to a set of $k $colors such that no hyperedge is monochromatic.
And for a hypergraph $\cH$, the chromatic number $\chi(\cH)$ of
$\cH$ is the minimum colors needed for proper coloring $\cH$.
\begin{theorem}[\cite{gerbner2025hypergraph}]\label{thm: zhou and Ge}
        Let $r \geq 2 $, $s \geq 1$ and $\chi(\mathcal{H})> 2$.
    \begin{align*}
        ex_r(n, \{\mathcal{H},M^{(r)+}_{s+1}\})=\binom{n}{r}-\binom{n-s}{r}-\binom{s}{r}+ ex_r(s,\mathcal{A}(\mathcal{H})).
    \end{align*} 
\end{theorem} 
In the same paper, they proposed the following conjecture.
\begin{conjecture}[\cite{gerbner2025hypergraph}]\label{conj: zhou and Ge}
    For a graph $G$ with $\chi(G)=\ell >r$ and assume that there is an independent set $U$ of $G$ such that deleting $U$ from $G$ results in a graph $G_0$ if chromatic number $\ell-1$, and there are two color classes of $G_0$ such that there are $m'(G)$ edges between them. If $s$ is sufficiently large, then
    $$\ex_r(n,\{G^{(r)+},M_{s+1}^{r+}\})=(m(G')-1)\binom{n}{r-1}+(s-m'(G)+1)\binom{\ell-2}{r-1}\left(\frac{n}{\ell-2}\right)^{r-1}+o(n^{r-1}).$$
\end{conjecture}

We also mention a recent hypergraph analogue of the Alon--Frankl theorem due to Yang, Zeng, and Zhang~\cite{yang2025hypergraph}. 
Let $\cK_{\ell+1}^r$ denote the family of $r$-graphs $\cF$ with at most $\binom{\ell+1}{2}$ edges such that, for some $(\ell+1)$-set $K$, every pair $\{x,y\}\subseteq K$ is covered by an edge of $\cF$.
\begin{theorem}[\cite{yang2025hypergraph}]\label{thm:YYZZ-KrMs}
Fix integers $\ell \ge r \ge 3$ and $s\ge 1$. For sufficiently large $n$,
\[
\ex_r\left(n,\{\cK_{\ell+1}^r, M_{s+1}^{(r)+}\}\right)= s \cdot t_{r-1}(n-s,\ell-1).
\]
Moreover, the unique extremal $r$-graph is obtained by fixing an $s$-set $V_0$ and taking all $r$-edges that contain exactly one vertex from $V_0$ and whose remaining $r-1$ vertices form an edge of the complete balanced $(\ell-1)$-partite $(r-1)$-graph on the other $n-s$ vertices.
\end{theorem}
If $\cK_{\ell+1}^r$ is replaced by $K_{\ell+1}^{(r)+}$, then it is a directly extension of Theorem \ref{thm1}. 
And notice that $\chi(K_{\ell+1}^{(r)+})=2$, Theorem \ref{thm: zhou and Ge} can not work.
In particular,Yang, Zeng and Zhang \cite{yang2025hypergraph} proposed the following conjecture. 

\begin{conjecture}[\cite{yang2025hypergraph}]\label{conj:4.1}
Let $\ell \ge r \ge 3$ and $s \ge \binom{\ell}{2}$ be integers. For sufficiently large $n$, we have
$$\mathrm{ex}_r\left(n,\{K_{\ell+1}^{(r)+},M^{(r)+}_{s+1}\}\right) = s \cdot t_{r-1}(n-s,\ell-1).
$$
\end{conjecture}

In this paper, we first establish the following exact result for all small $s$, which disproves Conjecture \ref{conj:4.1} for $s<\frac{\ell^2-1}{2}$. 

\begin{theorem}\label{thm: small-s}
Fix integers $s$ and $\ell\geq r\geq3$. Let $n$ be sufficiently large.

\noindent
\textbf{(i)} If $\ell>r$ and $\binom{\ell}{2}\leq s<\frac{\ell^2-1}{2}$, then
\begin{eqnarray*}
\ex_r(n,\{K_{\ell+1}^{(r)+},M^{(r)+}_{s+1}\})=\binom{\ell}{2}\binom{n-\binom{\ell}{2}}{r-1},
\end{eqnarray*}
\noindent
\textbf{(ii)} If $\ell>r$ and $\binom{\ell-1}{2}+r\leq s<\binom{\ell}{2}$, then
\begin{eqnarray*}
\ex_r(n,\{K_{\ell+1}^{(r)+},M^{(r)+}_{s+1}\})=s\binom{n-s}{r-1},
\end{eqnarray*}
\noindent
\textbf{(iii)} If $2+\binom{\ell-1}{2}\leq s< \binom{\ell-1}{2}+r$ and $\ell\geq 2r+1$, then
\begin{eqnarray*}
\ex_r(n,\{K_{\ell+1}^{(r)+},M^{(r)+}_{s+1}\})= s\binom{n-s}{r-1}+\sum_{i=s-\binom{\ell-1}{2}+1}^{r}\binom{s-1}{i-1}\binom{n-s}{r-i}. 
\end{eqnarray*}

\noindent
\textbf{(iv)} If $2\leq t\leq \ell-2$ and $\ell+1-t+\binom{t}{2}\leq s< \ell-t+\binom{t+1}{2}$, then
\begin{eqnarray*}
\ex_r(n,\{K_{\ell+1}^{(r)+},M^{(r)+}_{s+1}\})= s\binom{n-s}{r-1}+t_2(s,\ell-t)\binom{n-s}{r-2}+O(n^{r-3}),
\end{eqnarray*}

\noindent
\textbf{(v)} If $s< \ell$, then
\begin{eqnarray*}
\ex_r(n,\{K_{\ell+1}^{(r)+},M^{(r)+}_{s+1}\})= \binom{n}{r}-\binom{n-s}{r}.
\end{eqnarray*}
\end{theorem}
Moreover, we determine the exact value of $\mathrm{ex}_r(n,\{K_{\ell+1}^{(r)+},M^{(r)+}_{s+1}\})$ for large $s$, which implies that Conjecture~\ref{conj:4.1} holds for large $s$.

\begin{theorem}\label{large s}
    For integers $r\geq3$, $\ell\geq r$, there exists $s_0=s_0(\ell,r)$ such that for $s\geq s_0(\ell,r)$ and sufficiently large $n$, we have
    $$\ex_r(n,\{K_{\ell+1}^{(r)+},M^{(r)+}_{s+1}\})=s\cdot t_{r-1}(n-s,\ell-1).$$
\end{theorem}
Notice that for $K_{\ell+1}$, we have $m'(K_{\ell+1})=1$, where $m'(K_{\ell+1})$ is defined as in Conjecture \ref{conj: zhou and Ge}.
Thus, Theorem \ref{large s} confirms Conjecture \ref{conj: zhou and Ge} for then case when $G$ is a clique.

One of the key steps  in the proof of Theorem \ref{thm: small-s} and Theorem \ref{large s} is to determine the rainbow Tur\'{a}n number for expansion of cliques. 
Before stating the result, we introduce several notations. 

For a series of $r$-graphs $\cH_1,\dots,\cH_k$ on the same vertex set $V$, for every $r$-set $E\subseteq V$, let the multiplicity of $E$ (denoted by $m(E)$) be the number of $r$-graphs among $\cH_1,\dots,\cH_k$ that contain $E$ as a hyperedge.
Let $\cG$ denote the $r$-graph on the same vertex set $V$ whose hyperedges are all $r$-sets with multiplicity at least one.

We say that $r$-graphs $\cH_1,\dots,\cH_k$ on the same vertex set of order $n$ contain \textit{rainbow} copy of $\cF$ if their union contains a copy of $\cF$ with each hyperedge belonging to a distinct $\cH_i$. 

For a $r$-graph $\cF$, the \textit{rainbow hyper-Tur\'an number} of $\cF$, which is denoted by $\exsumr(n,k,\cF)$, is the maximum value of $\sum_{i=1}^ke(\cH_i)$, among all $n$-vertex $r$-graphs $\cH_1,\dots, \cH_k$ that do not contain a rainbow copy of $\cF$.
When $r=2$, we write $\ex^{\sum}(n,k,H)$ instead of $\ex^{\sum}_2(n,k,H)$ for simplicity.
When $r=2$, Keevash, Saks, Sudakov and Verstra{\"e}te \cite{keevash2004multicolour} gave the following result concerning rainbow copies of cliques.
\begin{theorem}[\cite{keevash2004multicolour}]\label{thmL sudakov rainbow}
For $\ell \geq 2$, $k\geq \binom{r}{2}$ and sufficiently large $n$, we have when $\frac{\ell^2-1}{2}\leq k$,
$$\ex^{\sum}(n,k,K_\ell)=k\cdot t_r(n,\ell-1),$$
and when $\binom{r}{2}\leq k<\frac{\ell^2-1}{2}$,
$$\ex^{\sum}(n,k,K_\ell)= \left(\binom{r}{2}-1\right)\binom{n}{2}.$$
\end{theorem}
Recently,
Li, Ma and Zheng\cite{li2025multicolorturan} 
determine the value of $\ex^{\sum}(n,k,H)$, when $H$ is a color critical graph with chromatic number $\ell$, and $k\geq 2\frac{\ell-1}{\ell}e(H)$, which partly supports a conjecture of Keevash, Mubayi, Sudakov and Verstra{\"e}te \cite{keevash2007rainbowturan}.
Moreover, Chakraborti, Kim, Lee, Liu and Seo \cite{chakraborti2024rainbowextremal} determined the value of $\ex^{\sum}(n,k,H)$, for almost all $r$-color-critical graphs when $r>4$.
Very recently, Gerbner and Miao \cite{gerbner2025matchingany} studied rainbow Tur\'an problems for forbidding a matching and an arbitrary graph.

In this paper, we extend the result of Theorem \ref{thmL sudakov rainbow} to the case when $r>2$.
\begin{theorem} \label{thm: k small rainbow}
    For integers $\ell\geq r\geq 3$, $1\leq k<\frac{\ell^2-1}{2}$ and sufficiently large $n$, we have
    $$\exsumr(n,k,K_\ell^{(r)+})= \min\left\{k,\left(\binom{\ell}{2}-1\right)\right\}\binom{n}{r}.$$
 \end{theorem}

Complementing this, we also characterize the case where $k$ is large:

 \begin{theorem}\label{thm: k large rainbow}
	For integers $\ell\geq r\geq 3$, there exists a constant $k_0=k_0(r,\ell)$ such that for all $k\geq k_0$ and sufficiently large $n$, we have
	$$\exsumr(n,k,K_\ell^{(r)+})= k\cdot t_{r}(n,\ell-1).$$
\end{theorem}
We also give some stability results about rainbow hyper-Tur\'an number, which is shown in Section~3.

The proof of Theorem \ref{thm: k small rainbow} and Theorem \ref{thm: k large rainbow} is useful for the proof of Theorem \ref{thm: small-s} and Theorem \ref{large s}.
We also make the statement more complete for intermediate ranges of parameters.


The paper is organized as follows:
In Section~2, we introduce notation and collect preliminary results.
In Section~3, we prove Theorems~\ref{thm: k small rainbow} and~\ref{thm: k large rainbow} on the rainbow Tur\'an number of $K_\ell^{(r)+}$, which will be used as key ingredients later.
In Section~4, we prove Theorem~\ref{thm: small-s}.
Finally, in Section~5, we prove Theorem~\ref{large s}.

\section{Notations and Preliminaries}
Let $\cH=(V(\cH),E(\cH))$ be an $r$-graph, and let $U\subseteq V(\cH)$ be a nonempty subset. 
We write $\cH-U$ for the $r$-graph obtained from $\cH$ by deleting the vertices in $U$ and all hyperedges incident to them, and $\cH[U]$ for the subhypergraph of $\cH$ contained in $U$. For a set $S\subseteq V(\cH)$, let $d_\cH(S)$ denote the number of hyperedges in $E(\cH)$ that contain $S$. 
When $S=\{v\}$, we write $d_\cH(v)$ instead of $d_\cH(\{v\})$ for simplicity.

Let $\cH_1,\cdots,\cH_k$ be a family of $r$-graphs defined on the same vertex set $V$. For every $r$-set $E\subseteq V$, let the multiplicity of $E$ (denoted by $m(E)$) be the number of $r$-graphs among $\cH_1,\dots,\cH_k$ that contain $E$ as a hyperedge.
Let $\cG$ denote the $r$-graph on the same vertex set $V$ whose hyperedges are all $r$-sets with multiplicity at least one.
For a set of hyperedges $\cS$, let $m(\cS)=\sum_{E\in \cS}m(E)$ denote the sum of the multiplicities of the hyperedges in $\cS$.
For a vertex $v \in V(\cG)$, the multiplicity of $v$ is defined as
$$
m(v)=\sum_{\substack{E \in \cG , E \ni v}} m(E),
$$
which is the sum of the multiplicities of all hyperedges in $\cG$ that contain $v$.

When $r=2$, Keevash, Saks, Sudakov and Verstra{\"e}te \cite{keevash2004multicolour} gave the following result concerning rainbow copies of cliques.

\begin{lemma}[\cite{keevash2004multicolour}]\label{lem: rainbow clique with weight function}
    Let $H_1,H_2,\dots,H_k$ be graphs on the same vertex set of order $n$, and $\ell\geq 3$. If there is a copy of $K_\ell$ with vertex set $v_1,\dots,v_\ell$, such that $m(v_1v_2)\geq \binom{\ell}{2}$, and for every $i\geq 3$, 
    $$\sum_{j<i}m(v_jv_i)\geq (i-1)\left(\binom{\ell}{2}-1\right),$$
    then there exists a rainbow copy of $K_\ell$.
\end{lemma}

Pikhurko~\cite{pikhurko2013exact} showed that $\operatorname{ex}_r(n, K_{\ell}^{(r)+})=t_r(n, \ell-1)$ with the unique extremal $r$-graph being $T_r(n, \ell-1)$. Pikhurko further established structural stability of near-extremal $K_{\ell}^{(r)+}$-free $r$-graphs.

\begin{theorem}[\cite{pikhurko2013exact}]\label{lem: stability Turan expansion}
    For every $\epsilon>0$, there exists $\delta=\delta(\epsilon,\ell,r)>0$ and $n_0$ such that if $\cH$ is an $r$-graph on $n\geq n_0$ vertices with no copy of $K_\ell^{(r)+}$ and
    $$|\cH|\geq t_r(n,\ell-1)-\delta n^r,$$
    then $\cH$ differs from the Tur\'an hypergraph $T_r(n,\ell-1)$ in at most $\epsilon n^r$ hyperedges.
\end{theorem}
\begin{lemma}\label{lem: find large matching}
     For integer $t\geq 1$, and $K_{t,t}$ is complete bipartite graph with two parts $A=\{a_1,\dots,a_t\}$, and $B=\{b_1,\dots,b_t\}$.
     And $\omega:E(K_{t,t})\to \mathbb{R}$ is a weighted function. For every $s\in \mathbb{R}$, if
for every $i\in[t]$, $\sum_{j\in [t]}\omega(a_ib_j)\geq s$, then there exists a matching $M$ with size $t$ such that
$\sum_{e\in M}\omega (e)\geq s$.
 \end{lemma}
 \begin{proof}
 First, set $G_t=K_{t,t}$ be the original graph.
Let $e_t$ be the edge with maximum weight $x_t$ among $E(G_t)$. We may assume $e_t=a_tb_t$.
Then set $G_{t-1}=G_t[\{a_1,\dots,a_{t-1}\cup\{b_1,\dots,b_{t-1}\}\}]$. And let $e_{t-1}$ be the edge with maximum weight $x_{t-1}$ among $E(G_{t-1})$.
Repeat this process, define
$$
G_{t-(i+1)} = G_{t-i}\bigl[\{a_1,\dots,a_{t-(i+1)}\}\cup\{b_1,\dots,b_{t-(i+1)}\}\bigr],
$$
and set $e_{t-(i+1)}$ to be the edge with maximum weight $x_{t-(i+1)}$ among $E(G_{t-(i+1)})$ (we may assume $e_{t-(i+1)}=a_{t-(i+1)} b_{t-(i+1)}$) for $i=1,\dots,t-1$.
Then $e_1=a_1b_1$ with weight $x_1$.
Then, we have $\omega(a_1b_j)\leq x_j$ for every $j\geq 1$, otherwise, it contradicts to the choice of $e_{j}$. Then we have
$$\sum_{i=1}^t x_i\geq \sum_{j=1}^t \omega(a_1b_j)\geq s,$$
and $M=\{e_1,\dots,e_t\}$ is the matching we search.
\end{proof}

We also need the following result on the Tur\'{a}n number of sunflowers, which is a special family of hypergraphs.
A family $A_1, \ldots, A_k$ of distinct sets is said to be a \textit{sunflower} if there exists a kernel $C$ contained in each of the $A_i$ such that the petals $A_i \backslash C$ are disjoint. In particular, we seek the $r$-uniform sunflower with $k$ petals and kernel of size $t$, which is denoted by $\mathcal{S}_t^{(r)}(k)$. Bradač, Bucić and Sudakov \cite{bradavc2023turan} determined the order of magnitude of the Turán number of $\mathcal{S}_t^{(r)}(k)$.

\begin{theorem}[\cite{bradavc2023turan}]\label{thm: sunflower}
    For integers $r\geq 2$, $t\geq 1$ and $k\geq 2$, and sufficiently large $n$, we have
$$
\operatorname{ex}(n, \mathcal{S}_t^{(r)}(k)) \approx_r \begin{cases}n^{r-t-1} k^{t+1} & \text { if } t \leq \frac{r-1}{2}, \\ n^t k^{r-t} & \text { if } t>\frac{r-1}{2}.\end{cases}
$$
\end{theorem}

\section{Proof of Theorem \ref{thm: k small rainbow}}

In this section, we assume that the family of $r$-graphs $\cH_1,\cdots,\cH_k$ on the same vertex set $V$ is rainbow $K_{\ell}^{(r)+}$-free and attains the maximum value of $\sum_{i=1}^k e(\cH_i)$. Let $\cG$ be the $r$-graph whose hyperedges are all $r$-sets with multiplicity at least one among $\cH_1,\dots,\cH_k$.

For this problem, there are two natural constructions. On the one hand, if $k \geq \frac{\ell^2-1}{2}$, then we take $k$ identical copies of the Tur\'an hypergraph $T_r(n,\ell-1)$. On the other hand, when $\binom{\ell}{2} \le k<\frac{\ell^2-1}{2}$, it is better to take $\binom{\ell}{2}-1$ copies of the complete $r$-graph and let the remaining hypergraphs have empty hyperedge sets, and of course for $k\le \binom{\ell}{2}-1$, it is optimal to take all $k$ hypergraphs to be complete $r$-graph. Thus, we may assume that $\binom{\ell}{2}\leq k <\frac{\ell^2-1}{2}$ in the following.

First we claim that it suffices to prove Theorem~\ref{thm: k small rainbow} for $r$-graphs that have $n>r^3$ vertices and the minimum multiple of vertices in $\cG$ is at least $\left(\binom{\ell}{2}-1 \right)\binom{n-1}{r-1}$, i.e.
\begin{equation}\label{eq: mini multi}
	\min_{v\in V(\cG)}m(v)\geq \left(\binom{\ell}{2}-1 \right)\binom{n-1}{r-1}.
\end{equation}
Indeed, suppose we have done this, and let $\cG$ with $n>r^8$ vertices, $e(\cG) \geqslant\left(\binom{\ell}{2}-1\right)\binom{n}{r}$ and no rainbow copy of $K_\ell^{(r)+}$. If the minimum multiple of vertices in $\cG$ is at least $\left(\binom{\ell}{2}-1 \right)\binom{n-1}{r-1}$, then we are done. 

Otherwise we obtain a contradiction as follows. Let $\cG=\cG(n), \cG(n-1), \ldots$ be a sequence of $r$-graphs where $\cG(m)$ has $m$ vertices and is obtained from $\cG(m+1)$ by deleting a vertex with multiplicity strictly less than $(\binom{\ell}{2}-1) \binom{m}{r-1}$ and all hyperedges incident to it. Setting $f(m)=e(\cG(m))-\left(\binom{\ell}{2}-1\right)\binom{m}{r}$ we have $f(n) \geqslant 0$ and $f(m) \geqslant f(m+1)+1$. If we can continue this process to obtain an $r$-graph $\cG(\ell^3)$, then
$$
n-\ell^3 \leqslant \sum_{m=\ell^3}^{n-1}(f(m)-f(m+1)) \leqslant f(\ell^3)<k\binom{\ell^3}{r}<\frac{\ell^2}{2}\binom{\ell^3}{r},
$$
which is a contradiction for $n>\ell^8$. Otherwise we obtain an $r$-graph $\cG(n^{\prime})$ with $n>n^{\prime}>\ell^3$ having the minimum multiple of vertices in $\cG(n^{\prime})$ is at least $\left(\binom{\ell}{2}-1\right)\binom{n^{\prime}-1}{r-1}$, containing no rainbow copy of $K_\ell^{(r)+}$ and $e(\cG(n^{\prime}))>(\binom{\ell}{2}-1)\binom{n^{\prime}}{r}$, which contradicts our assumption.

For every set $T$ with $t$ vertices, and an $(r-1)$-subset $A\subseteq V(\cG)\setminus T$, we call the set $A$ {\em fits} $T$ if the sum of multiplicities of hyperedges containing $A$ and one vertex in $T$, i.e.
$\sum_{|E\cap T|=1,A\subseteq E}m(E)$, is at least $t\left( \binom{\ell}{2}-1\right)$. Let $\cH(T)$ be the collection of $(r-1)$-subsets of $V(\cG)\setminus T$ that fit $T$.
\begin{lemma}\label{lem: sumflower}
    For every set $T$ with $t$ vertices, and a set $B\subseteq V$ with size $b$, where $b$ is a constant. The $(r-1)$-graph $\cH(T)$ contains a sunflower with one core in $V\setminus T$ and $t$ petals contained in $V\setminus T$ avoiding $B$.
\end{lemma}
\begin{proof}
    By the assumption on the minimum degree of $\cG$, we have
    $$\sum_{x\in T}m(x)\geq t \left(\binom{\ell}{2}-1\right)\binom{n-1}{r-1}.$$
    Let $\cE_i(T)$ be the set of hyperedges in $\cG$ containing exactly $i$ vertices in $T$.
    Then,
    $$\sum_{x\in T}m(x)=\sum_{i=1}^{t}i\cdot m(\cE_i(T))=m(\cE_1(T))+O(n^{r-2}).$$
    Let $\overline{m}$ denote the average multiplicity of hyperedges in $\cE_1(T)$, then
    \begin{equation}\label{eq: average multi, 1}
        \overline{m}\geq\frac{t (\binom{\ell}{2}-1)\binom{n-1}{r-1}-O(n^{r-2})}{\binom{n-t}{r-1}}>t\left(\binom{\ell}{2}-1\right)-\frac{1}{2}
    \end{equation}
    when $n$ is sufficiently large.
    Thus, there exists $\Theta(n^{r-1})$ hyperedges in $\cE_1(T)$ with multiplicity at least $t(\binom{\ell}{2}-1)$, and equally, there exists $\Theta(n^{r-1})$ $(r-1)$-sets contained in $V\setminus T$ that fit $T$.

    Since the number of $(r-1)$-sets intersecting $B$ is $O(n^{r-2})$, when $n$ is sufficiently large, the number of hyperedges in $\cE_1(T)$ fit $T$ and avoiding $B$ is also $\Theta(n^{r-1})$.
    By Theorem~\ref{thm: sunflower}, the Tur\'an number of sunflower with one core and $t$ petals is $O(n^{r-2})$ when $n$ is large enough. Thus, there exists an $(r-1)$-graph $\cH(T)$ contains a sunflower with one core in $V\setminus T$ and $t$ petals contained in $V\setminus T$ avoiding $B$.
\end{proof}

\noindent
\textbf{Proof of Theorem \ref{thm: k small rainbow}}.
We finish the proof by showing that every hyperedge is contained in at most $\binom{\ell}{2}-1$ hypergraphs.
\begin{lemma}\label{lemma: no large multi edge}
    For each hyperedge $E\in E(\cG)$, the multiplicity $m(E)\leq \binom{\ell}{2}-1$.
\end{lemma}
\begin{proof}
Suppose, for a contradiction, that there exists a hyperedge $E\in E(\cG)$ with multiplicity at least $\binom{\ell}{2}$.
Suppose $v_1,v_2\in E$. Then, by Lemma~\ref{lem: sumflower}, let $T=\{v_1,v_2\}$, there exists a sunflower in $\cH(\{v_1,v_2\})$ with one core and two petals $P_1^2,P_2^2$ avoiding $E$. And suppose $v_3$ as the core of this sunflower. Then we construct a complete bipartite graph $K_{2,2}$ with sets $A=\{v_1,v_2\}$ and $B=\{P_1^2,P_2^2\}$.
And the weight of the edge $v_iP_j^2$ (denoted by $\omega(v_iP^2_j)$) is the multiplicity of the hyperedge $\{v_j\cup P_j^2\}$.
According to the definition of $\cH(\{v_1,v_2\})$, we have $\sum_{j\in [2]}\omega(v_iP^2_j)\geq 2\left(\binom{\ell}{2}-1 \right)$.
Then, according to Lemma \ref{lem: find large matching}, there is a matching (we may assume it is $v_1P^2_1,v_2P^2_2$) with total weight $2\left(\binom{\ell}{2}-1 \right)$.
Then, we add hyperedges $v_iP^2_i$ to the edge $E$, and form an expansion clique $K_3^{(r)+}$.

Suppose we have found a copy of $K_{i}^{(r)+}$ containing $E$ with core vertices $v_1,v_2,\dots,v_i$.
Then, with a similar argument, there is a sunflower with $i$ petals $P^i_1,\dots,P^i_i$ and core $v_{i+1}$ avoiding $T\cup K_{i}^{(r)+}$ by Lemma \ref{lem: sumflower}.
Then, we construct a complete bipartite graph $K_{i,i}$ as above, and find a matching (we may assume it is $v_jP^i_j$ for $j\in[i]$, with total multiplicity $i\left( \binom{\ell}{2}-1 \right)$.
We do this for $\ell-1$ steps, and find a copy of $K_{\ell}^{(r)+}$ with $\{v_1,\dots,v_{\ell}\}$.
And suppose $E_{i,j}$ is the hyperedges containing $v_i,v_j$ for $2\leq i<j\leq \ell$.
The above process ensures $\sum_{j<i}m(E_{i,j})\geq (i-1)\left( \binom{\ell}{2}-1\right)$ for $i\in[\ell]$, and $E_{1,2}\geq \binom{\ell}{2}$.
Note that here $E_{i,j}$ are both hyperedges, and we can construct a graph $C$ with vertices $v_1,v_2,\dots,v_{\ell}$ and edges $e_{i,j}$ for $2\leq i<j\leq \ell$, with weight function $\omega'(e_{i,j})=m(E_{i,j})$.
And we have $\sum_{j<i}\omega'(e_{i,j})\geq (i-1)\left( \binom{\ell}{2}-1\right)$.
According to Lemma \ref{lem: rainbow clique with weight function},
there is a rainbow copy of $K_\ell$ in $C$, then 
there is a rainbow copy of $K_\ell^{(r)+}$ among the hyperedges $\{E_{i,j}\}_{2\leq i<j\leq \ell}$, which is a contradiction.
\end{proof}

Above Lemma implies that
$\sum_{i\in [k]}e(\cH_i)=\sum_{E\in \cG}m(E)\leq \left(\binom{\ell}{2}-1\right)\binom{n}{r}$.
This completes the proof.
\hfill
$\square$

Here, the above proof implies the following result, which is useful in the proof of Theorem \ref{thm: small-s}.
We call $\cH_1,\dots,\cH_k$ contains a \textit{super rainbow} copy of $\cF$, if for every fixed $i\in[k]$, there is an injection $\phi_i:E(\cF)\to [k]\setminus\{i\}$, such that for every $E\in E(\cF)$, $E\in \cH_{\phi_i(E)}$.

\begin{cor}\label{cor: super rainbow}
Let $3\leq r\leq \ell$ be integers, and $\binom{\ell}{2}<k<\frac{\ell^2-1}{2}$.
Suppose $\cH_1,\dots,\cH_k$ are $r$-uniform graphs on the same vertex set $V$.
If $\sum_{i=1}^ke(\cH_i)> \binom{\ell}{2}\binom{n}{r}$, then $\cH_1,\dots,\cH_k$ contains a super rainbow copy of $K_{\ell}^{(r)+}$.
\end{cor}
\begin{proof}
    The proof is similar to the above proof. First, the inequality of the minimum multiple of vertices in $\cG$ is changed from (\ref{eq: mini multi}) to
    $$\min_{v\in V(\cG)} m(v)\geq \binom{\ell}{2}\binom{n-1}{r-1}.$$
    For a $t$-set $T$,
    And we redefine the $(r-1)$-set $A$ fits $T$ if the sum of the multiplicity of hyperedges containing $A$ and one vertex in $T$ is at least $t\binom{\ell}{2}$.
    Then, Lemma \ref{lem: sumflower} still holds.

    Since $\sum_{i=1}^ke(\cH_i)> \binom{\ell}{2}\binom{n}{r}$, there is a hyperedge $E$ with multiplicity at least $\binom{\ell}{2}+1$. Suppose $\{v_1,v_2\}\subseteq E$.
    With the same search process as in Lemma \ref{lemma: no large multi edge}, there is a copy of $K_{\ell}^{(r)+}$ (denoted by $\cK$) containing $E$, with core vertices $\{v_1,v_2,\dots,v_\ell\}$, and for $i\geq 2$,
    $$\sum_{j<i}m(E_{i,j})\geq (i-1)\binom{\ell}{2},$$
    where $E_{i,j}$ is the hyperedge in $\cK$ containing $v_i,v_j$.
    Then, for every $i\in [k]$, let $m_1(E)$ denote the number of hyperedge among $\{\cH_j\}_{j\neq i}$ that contains $E$.
    Then, $m_1(E)\geq m(E)-1$.

    As a result, $m_1(E)\geq \binom{\ell}{2}-1$, and for $i\geq 2$,
    $$\sum_{j<i}m_1(E_{i,j})\geq (i-1)\left(\binom{\ell}{2}-1\right).$$
    Then, similarly according to Lemma \ref{lem: rainbow clique with weight function}, $\cK$ is a rainbow copy of $K_\ell^{(r)+}$ of hypergraphs $\{\cH_j\}_{j\neq i}$. This completes the proof.
\end{proof}

Here, we also give a stability result.

\begin{theorem}\label{thm : stability of rainbow}
    Suppose $\cH_1,\dots,\cH_k$ are $r$-graphs on $n$-vertex set $V$ with no rainbow copy $K_\ell^{(r)+}$, where $\binom{\ell}{2}-1 \leq k<\frac{\ell^2-1}{2}$, and $|\cH_1|\geq \dots\geq |\cH_k|$.
    If $\sum_{i\in[k]}e(\cH_i)\geq (\binom{\ell}{2}-1)\binom{n}{r}-O(n^{r-1})$, then we can delete $O(n^{r-1})$ hyperedges such that all remaining hyperedges is contained exactly in $\cH_1,\dots,\cH_{\binom{\ell}{2}-1}$.
\end{theorem}
\begin{proof}
    According to the lower bound on $\sum_{i\in[k]}e(\cH_i)$, (\ref{eq: mini multi}) still holds. Then, Lemma \ref{lem: sumflower} and Lemma \ref{lemma: no large multi edge} hold.
    This implies every hyperedge is contained in at most $\binom{\ell}{2}-1$ hypergraph of $\cH_1,\dots,\cH_k$.

    Let $\cH$ denote the hyperedges with multiplicity at most $\binom{\ell}{2}-2$. Then
    $$\left(\binom{\ell}{2}-1\right)\binom{n}{r}-O(n^{r-1})\leq\sum_{i\in[k]}e(\cH_i)\leq |\cH|\cdot\left(\binom{\ell}{2}-2 \right)+\left( \binom{n}{r}-|\cH|\right)\cdot\left(\binom{\ell}{2}-1\right).$$

    Then we have $|\cH|=O(n^{r-1})$.
    Collect all the hyperedges with multiplicity $\binom{\ell}{2}-1$ as $\cF$.

    \begin{claim}\label{claim: cliques has same list}
        For every copy of $K_\ell^{(r)+}$ in $\cF$, we claim that all the hyperedges in that copy are contained in the same $\binom{\ell}{2}-1$ hypergraphs.
\end{claim}
    \begin{proof}
    Otherwise, suppose that $E_1$ and $E_2$ are contained in different collections of $\binom{\ell}{2}-1$ hypergraphs. We say a hyperedge contains colour $c$ if it is contained in $\cH_c$. Its colour list is the index set $I$ such that $E\in \cH_c$ for every $c\in I$.
    Next, we prove we can assign $\binom{\ell}{2}$ colours to the copy $K_{\ell}^{(r)+}$, which implies there is a rainbow $K_{\ell}^{(r)+}$.

    We can greedily choose colours for each hyperedge, and $E_1$ is the last hyperedge to choose a colour.
    If the colour list of $E_1$ contains a colour not used for the other $\binom{\ell}{2}-1$ hyperedges, then we can colour $E_1$ with this colour, and find a rainbow copy of $K_{\ell}^{(r)+}$.

    Suppose all colours in the list of $E_1$ have been used for all the other hyperedges. Then, since $E_1$ and $E_2$ have different colour lists,
     there is a colour $c$ in the list of $E_2$ that has not been used for the other $\binom{\ell}{2}-1$ hyperedges besides $E_1$. And let $c'$ be the current colour of $E_2$.
    Then we can change the colour of $E_2$ to $c$ and colour $E_1$ to $c'$. This yields a rainbow $K_{\ell}^{(r)+}$, a contradiction. Therefore, the claim holds.
    \end{proof}

    Notice that $|\cF|=\binom{n}{r}-O(n^{r-1})$. For two vertices $u$ and $v$, we call them a \textit{couple} if $d_{\cF}(\{u,v\})\geq C(n,\ell)n^{r-3}$, where $C(r,\ell)=2\binom{\ell+1}{2}r$ is a constant depending on $r$ and $\ell$. We construct an auxiliary graph $H_\cF$ on vertex set $V(\cF)$, where two vertices $u,v$ are adjacent if $u,v$ are a couple.
    If $u,v$ are not a couple, then at least $\binom{n}{r-2}-C(r,\ell)n^{r-3}$ hyperedges containing $\{u,v\}$ are not in $\cF$. Since there are only $O(n^{r-1})$ $r$-sets in $\binom{V(\cF)}{r}$ that are not in $\cF$, it follows that only $O(n)$ pairs in $\binom{V(\cF)}{2}$ are not couples. Consequently, only $O(n)$ edges are missing from $H_\cF$, and hence $e(H_\cF)\geq \binom{n}{2}-O(n)$.

    Let $A=\{v\in V(H_\cF):d_{H_\cF}(v)\geq \frac{\ell+4}{\ell+5}n\}$.
    Then, we have $|A|=n-O(1)$.
    And let $\cF'=\{E\in \cF:|E\cap A|\geq 2\}$.
    \begin{claim}
        All the hyperedges in $\cF'$ are contained in the same $\binom{\ell}{2}-1$ hypergraphs.
    \end{claim}
    \begin{proof}
        We only need to prove for every two disjoint $E_1,E_2\in \cF'$, they are contained in the same $\binom{\ell}{2}-1$ hypergraphs. Then, since for every intersecting hyperedges $E_1,E_2$, there exists a hyperedge $E_3$ is disjoint with both $E_1$ and $E_2$, $E_1$ and $E_3$, $E_2$ and $E_3$ are both in the same $\binom{\ell}{2}-1$ hypergraphs, so are $E_1$ and $E_2$.

        Now suppose $E_1\cap E_2=\emptyset$, $u_1,u_2\in E_1\cap A$ and $w_1,w_2\in E_2\cap A$.
        By the definition of $H_\cG$, every $\ell+4$ vertices, the number of their common neighbours in $H_\cG$ is at least $$|A|-(\ell+4)\cdot\left(1-\frac{\ell+4}{\ell+5}\right)n\geq \frac{1}{\ell+6}n.$$
        Thus, by a greedy search, there exists a clique $K_{\ell-1}$ with vertices $v_1,\dots,v_{\ell-1}$ in the common neighbours of $\{u_1,u_2,w_1,w_2\}$.

        Since for every edge $uv\in E(H_\cG)$, $u,v$ is contained in at least $C(r,\ell)n^{r-3}$ hyperedges in $\cF$, there exists a clique $K_{\ell+1}^{(r)+}$ (denoted by $\cK_1$) contains $E_1$ and with core $\{u_1,u_2,v_1,\dots,v_{\ell-1}\}$, and a clique $K_{\ell+1}^{(r)+}$ (denoted by $\cK_2$) contains $E_2$ and with core $\{w_1,w_2,v_1,\dots,v_{\ell-1}\}$.

        When $\ell\geq 3$, since the number of hyperedges contains $v_1,v_2$ are at least $C(r,\ell)n^{r-3}$, we may assume that there is a hyperedge $E_3$ containing $v_1,v_2$ are contained in $E(\cK_1)\cap E(\cK_2)$.
        Then, according to Claim \ref{claim: cliques has same list}, $E_i$ and $E_3$ are contained in the same $\binom{\ell}{2}-1$ hypergraphs, for $i=1,2$, and so are $E_1,E_2$.
    \end{proof}

Notice that $e(\cF')\geq \binom{n}{r}-O(n^{r-1})$, which completes the proof.
\end{proof}

In the next part of this section, we deal with the case when $k$ is sufficiently large, relative to $\ell$. We restate Theorem \ref{thm: k large rainbow} here for convenience.

\noindent\textbf{Theorem.}
    \textit{For integers $\ell\geq r\geq 3$, there exists a constant $k_0=k_0(r,\ell)$ such that for all $k\geq k_0$ and sufficiently large $n$, we have}
    $$\exsumr(n,k,K_\ell^{(r)+})= k\cdot t_{r}(n,\ell-1).$$

\noindent\textbf{Proof of Theorem \ref{thm: k large rainbow}.}
Let $\cH_1,\dots,\cH_k$ be $n$-vertex $r$-graphs on the same vertex set $V$ that do not contain a rainbow $K_\ell^{(r)+}$, such that the value $\sum_{i=1}^ke(\cH_i)$ is maximized.
The lower bound is attained by taking $\cH_i\cong T_r(n,\ell-1)$ for all $i\in[k]$. Then we deal with the upper bound.

Let $\cG$ be the $r$-graph whose hyperedges are all $r$-sets with multiplicity at least one among $\cH_1,\dots,\cH_k$.
Let $m(\cG)=\sum_{E \in E(\cG)} m(E)$. Based on the lower bound, we may assume that $m(\cG) \geq k \cdot t_r(n, \ell-1)$.

We partition $E(\cG)$ into bad edges $\cF_1=\{E: m(E) \leq\binom{\ell}{2}-1\}$ and good edges $\cF_2=\{E: m(E)>\binom{\ell}{2}-1\}$.
It follows that
$$
m(\cG)\leq |\cF_1|\cdot \left(\binom{\ell}{2}-1\right) + |\cF_2|\cdot k\leq \binom{n}{r}\cdot \left(\binom{\ell}{2}-1\right) + |\cF_2|\cdot k.
$$
This inequality implies
$$
|\cF_2|\geq t_r(n,\ell-1)-\beta(k)n^{r},
$$
where $\beta(k)=\frac{\left(\binom{\ell}{2}-1\right)}{k\cdot r!}$ vanishes as $k \rightarrow \infty$.

Let $\epsilon>0$ be a sufficiently small constant to be determined later. We choose $k$ large enough so that $\beta(k) \leq \delta$, where $\delta=\delta(\epsilon, \ell, r)$ is the parameter provided by Lemma \ref{lem: stability Turan expansion}. The hierarchy of constants in this proof is thus: $\ell, r \rightarrow \epsilon \rightarrow \delta \rightarrow k_0$, ensuring that $k_0$ (and thus $k$) depends only on $\ell$ and $r$, and is independent of $n$.

Notice that $\cF_2$ has no copy of $K_\ell^{(r)+}$, otherwise we can find a rainbow $K_\ell^{(r)+}$.
According to Lemma~\ref{lem: stability Turan expansion}, there exists a partition of $V(\mathcal{G}) = V_1 \cup \dots \cup V_{\ell-1}$ such that $\mathcal{F}_2$ is $\epsilon n^r$-close to the complete $(\ell-1)$-partite $r$-graph $\mathcal{K}$ on these parts. Moreover, each part $V_i$ satisfies $\left|\, |V_i| - \frac{n}{\ell-1}\,\right| \leq \epsilon^{1/2}n$ for $i \in [\ell-1]$.

Let $\mathcal{E}_1$ and $\mathcal{E}_2$ denote the collections of hyperedges in $\mathcal{F}_1 \setminus \mathcal{K}$ and $\mathcal{F}_2 \setminus \mathcal{K}$ respectively. Similarly, let $\mathcal{M} = \mathcal{K} \setminus \mathcal{F}_2$ be the collection of missing edges in $\mathcal{F}_2$ relative to $\mathcal{K}$. By the stability property, we have $|\mathcal{E}_2| \leq \epsilon n^r$ and $|\mathcal{M}| \leq \epsilon n^r$. For each hyperedge $E \in \mathcal{K}$, we define $m'(E) = k - m(E)$ as the missing multiplicity of $E$. Our goal is to establish the following inequality:
\begin{equation}\label{eq: missing degree goal}
	\sum_{E \in \mathcal{E}_1 \cup \mathcal{E}_2} m(E) < \sum_{E \in \mathcal{M}} m'(E).
\end{equation}
If this holds, we can replace all hyperedges in $\mathcal{E}_1 \cup \mathcal{E}_2$ with missing edges from $\mathcal{K}$ across the collection $\{\mathcal{H}_i\}_{i=1}^k$ without decreasing the total number of edges.

Let $X = \{v \in V(\mathcal{G}) : d_{\mathcal{M}}(v) \geq 2\epsilon^{1/2} n^{r-1}\}$ be the set of vertices that are contained in a large number of missing hyperedges from $\mathcal{M}$. A simple counting argument yields $|X| \leq \frac{r|\mathcal{M}|}{2\epsilon^{1/2}n^{r-1}} \leq \frac{r}{2} \epsilon^{1/2} n$. For each $i \in [\ell-1]$, we define $V_i' = V_i \setminus X$ and $X_i = V_i \cap X$.

We say that two vertices $u_i \in V_i'$ and $u_j \in V_j'$ ($i \neq j$) are \textit{coupled} if they are contained in at least $r \binom{\ell}{2} n^{r-3}$ hyperedges of $\mathcal{F}_2 \cap \mathcal{K}$. Accordingly, we define an $(\ell-1)$-partite auxiliary graph $H_{\mathcal{G}}$ on $V_1' \cup \dots \cup V_{\ell-1}'$, where an edge exists between $u_i \in V_i'$ and $u_j \in V_j'$ if and only if they are coupled. For any $u \in V_i'$ and $j \neq i$, we denote by $C_j(u) = N_{H_{\mathcal{G}}}(u) \cap V_j'$ the set of vertices in $V_j'$ coupled with $u$. By construction, $d_{\mathcal{F}_2 \cap \mathcal{K}}(\{u,v\}) \geq r \binom{\ell}{2} n^{r-3}$ for every $v \in C_j(u)$.

\begin{claim}
    For every $u\in V_i'$ and $i\neq j$, the size of $C_j(u)$ satisfies
\end{claim}
    \begin{equation}\label{eq: size of Cj(u)}
        |C_j(u)|\geq \left(\frac{1}{\ell-1}-2 \cdot\epsilon^{\frac{1}{2}}\right)n.
    \end{equation}
        \begin{proof}
             For every $u_j\in V_j'\setminus C_j(u)$, the number of missing hyperedges in $\mathcal{M}$ containing the pair $\{u, u_j\}$ is at least 
             $$
             d_\cK(\{u,u_j\})-d_{\cF_2\cap \cK}(\{u,u_j\})\geq \binom{\ell-3}{r-2}\left(\frac{n}{2(\ell-1)}\right)^{r-2}-r \binom{\ell}{2}n^{r-3}>\left(\frac{n}{2(\ell-1)}\right)^{r-2}.
             $$ 
             Since $u\in V_i'$, we have $d_\cM(u)\leq 2\epsilon^{\frac{1}{2}} n^{r-1}$. Moreover, since there are no hyperedges in $\cM$ containing two vertices in $V_j$, the size of $V_j'\setminus C_j(u)$ is at most $\frac{2\epsilon^{\frac{1}{2}} n^{r-1}}{(\frac{n}{2(\ell-1)})^{r-2}}$.  
             Thus,
             $$|C_j(u)|=|N_{H_\cG}(u)\cap V^{\prime}_j|\geq |V_j'|-\frac{2\epsilon^{\frac{1}{2}} n^{r-1}}{(\frac{n}{2(\ell-1)})^{r-2}}\geq \frac{n}{\ell-1}-\epsilon^{\frac{1}{3}}n,$$
             when $\epsilon<\left(\frac{1}{r+1+(2(\ell-1))^{r-2}}\right)^{6}$.
\end{proof}

The preceding claim implies the following crucial property regarding the common neighbors in the auxiliary graph $H_{\mathcal{G}}$.

\begin{claim}\label{claim: finding common neighbour}
	Let $I \subset [\ell-1]$ be a set of indices and $S \subseteq \bigcup_{i \in I} V_i'$ be a set of vertices with $|S| < r \binom{\ell}{2}$. For every $j \notin I$, the size of the common neighborhood of $S$ in $V_j'$ in  graph $H_{\mathcal{G}}$ satisfies
	\begin{equation}
		\left| \bigcap_{u \in S} C_j(u) \right| \geq \left( \frac{1}{\ell-1} - 2r \binom{\ell}{2} \epsilon^{1/3} \right) n.
	\end{equation}
    \end{claim}
\begin{proof}
	For each $u \in S$, we have $|V_j' \setminus C_j(u)| \leq \epsilon^{1/3} n$. We obtain
	$$
	\left| V_j' \setminus \bigcap_{u \in S} C_j(u) \right| = \left| \bigcup_{u \in S} (V_j' \setminus C_j(u)) \right| \leq \sum_{u \in S} |V_j' \setminus C_j(u)| < r \binom{\ell}{2} \epsilon^{1/3} n.
	$$
	Since $|V_j'| \geq \frac{n}{\ell-1} - (r+1)\epsilon^{1/2} n$, we have $|\bigcap_{u \in S} C_j(u)|\geq \left(\frac{1}{\ell-1} - 2r \binom{\ell}{2} \epsilon^{1/3}\right) n$.
\end{proof}

\begin{claim}\label{claim: intersect with one vertex}
    For every hyperedge $E \in \mathcal{E}_1 \cup \mathcal{E}_2$, we have $|E \cap V_i'| \leq 1$ for each $i \in [\ell-1]$.
\end{claim}
\begin{proof}
    Suppose, to the contrary, that there exists a hyperedge $E$ such that $|E \cap V_1'| \geq 2$. Let $u_0, u_1 \in E \cap V_1'$ be two distinct vertices. We shall show that there exists a copy of $K_\ell^{(r)+}$ in $\mathcal{G}$ using the hyperedge $E \in \mathcal{F}_1 \cup \mathcal{F}_2$ and $\binom{\ell}{2}-1$ additional hyperedges from $\mathcal{F}_2$. 
    Since $E \in \mathcal{F}_1 \cup \mathcal{F}_2$ and all other $\binom{\ell}{2}-1$ hyperedges belong to $\mathcal{F}_2$ with multiplicity at least $\binom{\ell}{2}$, a simple greedy assignment of distinct graph indices yields a rainbow $K_\ell^{(r)+}$.

    We proceed by induction to construct a sequence of clique expansions. Suppose we have already found a copy of $K_i^{(r)+}$ containing $E$, with two core vertices $u_0, u_1 \in V_1'$ and $i-2$ additional core vertices $u_j \in V_j'$ for $j=2, \dots, i$, such that all hyperedges except $E$ belong to $\mathcal{F}_2$. 

    Then, we will enlarge this clique to $K_{i+1}^{(r)+}$. According to Claim \ref{claim: finding common neighbour}, there is a vertex $u_{i+1}$ in $V_{i+1}$ with $u_{i+1}\in \bigcap_{j=1}^i C_{i+1}(u_j)$.
    By the definition of $C_{i+1}(u_j)$, for every $u_j$, we have $d_{\cF_2\cap \cK}(\{u_j,u_{i+1}\})\geq r \binom{\ell}{2}n^{r-3}$, then
    we can greedily choose hyperedges in $\cF_2\cap \cK$ containing $u_{i+1}u_j$, and disjoint with other vertices in clique $K_i^{(r)+}$. This forms a copy of $K_{i+1}^{(r)+}$.
    Thus, by repeating this process, we can find a copy of $K_\ell^{(r)+}$ containing $E$.
    \end{proof}

For every $x\in X_i$, let $\cA_x$ denote the collection of hyperedges in $\cE_1\cup \cE_2$ containing $x$ and no other vertex in $X$.
And let $\cB_x$ denote the collection of hyperedges in $\cE_1\cup \cE_2$ containing $x$ and at least one other vertex in $X_i$.
By Claim \ref{claim: intersect with one vertex}, we have $\bigcup_{x\in X}(\cA_x\cup \cB_x)=\cE_1\cup \cE_2$.


It is clear that
\begin{equation}\label{eq: Bx}
    |\mathcal{B}_x| \leq |X| n^{r-2} \leq \frac{r}{2} \epsilon^{1/2} \left(\frac{1}{\ell-1}-\epsilon^{1/3}\right)^{r-2}n^{r-2}
\end{equation} for each $x \in X$. Let $\mathcal{B}$ be the collection of hyperedges in $\mathcal{E}_1 \cup \mathcal{E}_2$ that contain at least two vertices from $X$, and then we have the bound $|\mathcal{B}| \leq \binom{|X|}{2} n^{r-2} \leq |X|^2 n^{r-2}$. 

To complete the structural analysis, first we suppose that there is a hyperedge $E\in \cA_x$ containing $x \in X_1$ and $v_1 \in V_1'$.

    

\begin{claim}\label{claim: not all in Ci(x)}
	For every copy of $K_{\ell-1}^{(r)+}$ with core vertices $\{v_1, u_2, \dots, u_{\ell-1}\}$ where $v_1 \in V_1'$ and $u_i \in V_i'$, and it intersects $E$ with $v_1$. Then, there exists at least one vertex $u_i$ such that $u_i \notin C_i(x)$, which implies $d_{\mathcal{F}_2 \cap \mathcal{K}}(\{x, u_i\}) < r \binom{\ell}{2} n^{r-3}$.
\end{claim}
\begin{proof}
	Otherwise, an analogous greedy construction to that in the proof of Claim~\ref{claim: intersect with one vertex} yields a copy of $K_\ell^{(r)+}$ containing the hyperedge $E \in \mathcal{A}_x$. Since all hyperedges in this copy besides $E$ belong to $\mathcal{F}_2$ and thus have multiplicity at least $\binom{\ell}{2}$, a rainbow $K_\ell^{(r)+}$ can be formed. This contradicts the rainbow $K_\ell^{(r)+}$-free property of $\mathcal{G}$ and completes the proof.
\end{proof}

This implies that at least one of $i\in [\ell-1]\setminus \{1\}$ such that $|V_i'\setminus C_i(x)|\geq \frac{n}{2(\ell-1)}$. Otherwise, according to Claim \ref{claim: finding common neighbour}, we can find a clique $K_{\ell-1}^{(r)+}$ containing $v_1$ and avoiding other $r-1$ vertices in $E$, and with core vertices in $C_i(x)$. A contradiction with Claim \ref{claim: not all in Ci(x)}.

Without loss of generality, we assume $|V_2' \setminus C_2(x)| \geq \frac{n}{2(\ell-1)}$. For every $u_2 \in V_2' \setminus C_2(x)$, we consider the sum of the missing multiplicities of hyperedges in $\mathcal{K}$ containing both $u_2$ and $x$. It follows that
$$
\sum_{E \in \mathcal{K}, \{u_2,x\} \subseteq E} m'(E) \geq \left( k - \binom{\ell}{2} + 1 \right) \left( d_{\mathcal{K}}(x, u_2) - r \binom{\ell}{2} n^{r-3} \right) \geq k \cdot c(r,r) n^{r-2}
$$
for some constant $c(r,\ell)$ provided that $k > 2\binom{\ell}{2}$. Summing over all choices of $u_2 \in V_2' \setminus C_2(x)$, the total missing multiplicity of hyperedges containing $x$ satisfies
$$
\sum_{u_2 \in V_2' \setminus C_2(x)} \sum_{E \in \mathcal{K}, \{u_2,x\} \subseteq E} m'(E) \geq \frac{k \cdot c(r,\ell)}{2(\ell-1)} n^{r-1}.
$$

We may assume that the partition $V_1,\dots,V_{\ell-1}$ is chosen such that the total sum of missing multiplicities of hyperedges containing vertices in $X$ is minimized.
Notice for $x\in X_1$, every $E\in \cA_x$ contains at most one vertex in $V_i$ for $i\geq2$, and there exists $j\geq 2$ such that $E\cap V_j=\emptyset$.
Then for every $i\neq 1$, we have
\begin{equation}\label{eq: move vertex}
    \sum_{x\in E\in \cK\cap \cF, E\cap V_i\neq \emptyset}m(E)\geq \sum_{E\in \cA_x, E\cap V_i=\emptyset}m(E)-\sum_{E\in \cB_x}m(E).   
\end{equation}
If this inequality were violated, moving $x$ from $V_1$ to $V_i$ would decrease the total missing multiplicity, which contradicts the minimality of our chosen partition.

Let $\mathcal{A}_x^{\prime}$ denote the collection of hyperedges in $\mathcal{A}_x$ with multiplicities at least $\binom{\ell}{2}$.
\begin{claim}\label{clm: extral degree total}
    We have 
    $$\sum_{E\in \cA_x'}m(E)\leq \frac{k\cdot c(r,\ell)}{4(\ell-1)} n^{r-1}.$$
\end{claim}
\begin{proof}
    Suppose, for a contradiction, that $\sum_{E\in \cA_x'}m(E)>\frac{k\cdot c(r,\ell)}{4(\ell-1)} n^{r-1}$.
Notice that each hyperedge in $\cA_x'$ intersects with exactly $r-2$ different parts in $V_2,\dots,V_{\ell-1}$.
Thus, without loss of generality, we assume that the sum of multiplicities of hyperedges in $\cA_x'$ intersecting with $V_2,\dots,V_{r-1}$ is at least $$\frac{1}{\binom{\ell-2}{r-2}} \sum_{E\in \cA_x'}m(E)\geq \frac{1}{\binom{\ell-2}{r-2}} \frac{k\cdot c(r,\ell)}{4(\ell-1)} n^{r-1}=\beta(r,\ell) k n^{r-1}.$$
And then, the number of hyperedges in $\cA_x'$ intersecting with $V_2,\dots,V_{r-1}$ (denoted by $\cA_{x,[r-1]}'$) is at least $\beta(r,\ell) n^{r-1}$.

Where $\beta(r,\ell)=\frac{1}{\binom{\ell-2}{r-2}}\frac{c(r,\ell)}{4(\ell-1)}$ is a constant.
This implies for every $ r-1\leq i\leq \ell-1$, $$\sum_{E\in \cA_x,E\cap V_i=\emptyset}m(E)\geq \beta(r,\ell) k n^{r-1}.$$

Then, according to \eqref{eq: move vertex}, for every $r-1\le i\le \ell-1$,
$$\sum_{x\in E\in \cK\cap \cF,E\cap V_i\neq \emptyset}m(E)\geq \beta(r,\ell) k n^{r-1}-k\cdot|\cB_x|\geq \beta(r,\ell) k n^{r-1}-k\cdot r\epsilon^{\frac{1}{2}} n^{r-1}.$$
Then the number of hyperedges in $\cK\cap \cF$ containing $x$ and intersecting with $V_i$ (collect the set of such hyperedges $\cK_i(x)$) is at least
$$\frac{1}{k}(\beta(r,\ell) k n^{r-1}-k\cdot \epsilon^{\frac{1}{2}} n^{r-1})\geq \frac{1}{2}\beta(r,\ell) n^{r-1}.$$

By the size of $\cA_x'$, there exists a vertex $v\in V_1$ such that the number of hyperedges in $\cA_x'$ containing both $x$ and $v$ is at least $\frac{1}{2}\beta(r,\ell) n^{r-2}$.
Recall that $C_i(x)$ is the vertex $v_i$ in $V_i$, with $d_{\cF_2\cap \cM}(\{v_i,x\})\geq C(r,\ell)n^{r-3}$. And according to the size of $\cK_i(x)$, for $r-1\le i\le \ell-1$,
$$\frac{1}{2}\beta(r,\ell) n^{r-1}\leq \cK_i(x)|\leq |V_i'\setminus C_i(x)|\cdot C(r,\ell)n^{r-3}+|C_i(x)|n^{r-2}.$$
It implies $|C_i(x)|\geq \frac{1}{3}\beta(r,\ell) n$ for $r-1\le i\le \ell-1$.
And since the number of hyperedges in $\cA_{x,[r-1]}'$ is at least $\beta(r,\ell) n^{r-1}$, then for each $V_i$ where $2\le i\le  
r-2$, we have
$$\beta(r,\ell) n^{r-1}\leq |\cA_{x,[r-1]}'| \leq |V_1'\setminus C_i(x)|\cdot C(r,\ell) n^{r-3}+|C_i(x)|n^{r-2}.$$
This implies $|C_i(x)|\geq \frac{1}{3}\beta(r,\ell) n$ for all $2\le i\le r-2$. And the same bound of $|C_i(x)|$ holds for all $2\le i\le \ell-1$.
 
According to Claim \ref{claim: finding common neighbour}, and with similar process as Claim \ref{claim: intersect with one vertex}, there is a copy of $K_{\ell-1}^{(r)+}$ containing $v$ with each part containing at most one core vertex (denoted by $u_i$), and $u_i\in C_i(x)$ for $i\neq 1$.
By the definition of $C_i(x)$, there are $\ell-2$ hyperedges $E_2,\dots, E_{\ell-1}\in \cF_2\cap\cM$, such that $E_i\cap E_j=\{x,u_i\}$ for $2\leq i<j\leq \ell-1$, and $E_i\cap V(K_{\ell-1}^{(r)+})=\{u_i\}$.
Since $d_{\cF_2}(\{v,x\})\geq \frac{1}{2}\beta(r,\ell) n^{r-2}$, there is a hyperedge $E_1\in \cF_2$ containing $\{v,x\}$ and avoiding other vertices in $V(K_{\ell-1}^{(r)+})\cup(\cup_{i=2}^{\ell-1}E_i)$.
Then, there is a copy of $K_{\ell}^{(r)+}$ with hyperedges $E(K_{\ell-1}^{(r)+})(\cup_{i=1}^{\ell-1} E_i)$ in $\cF_2$.
Then, there exists a copy of rainbow $K_{\ell}^{(r)+}$, a contradiction. 
\end{proof}
The above claim implies that
$$\sum_{E\in \cA_x'}m(E)\leq \frac{k\cdot c(r,\ell)}{4(\ell-1)}\cdot n^{r-1}\leq \frac{1}{2}\sum_{E\in \cM,x\in E}m'(E).$$
And we have $\sum_{E\in \cA_x\setminus \cA_x'}m(E)\leq (\binom{\ell}{2}-1)n^{r-1}<\frac{1}{6}\sum_{x\in E\in \cK}m'(E)$
when $k\geq \frac{ 12(\ell-1)(\binom{\ell}{2}-1)}{c(r,\ell)}$.
Now, we have proved that the sum of multiplicities of hyperedges in $\cA_x\cup \cB_x$ is at most $\frac{2}{3}\sum_{E\in \cM,x\in E}m'(E)$.

Moreover, since the sum of multiplicities of hyperedges in $\cB_x$ is at most
$$\epsilon k\cdot n^{r-2}<\frac{1}{6}|\cM_x|,$$
now we proved that $|\cA_x|+|\cB_x|<\frac{5}{6} \sum_{ E \in \mathcal{M},x\in E} m'(E)$.

And for $x\in X$, if there are no hyperedges $E$ containing both $x$ and some vertex $v_1'\in V_1$, then we have $\cA_x=\emptyset$. 
And by (\ref{eq: Bx}), the sum of multiplicity in $\cB_x$ is at most
$$k\frac{r}{2} \epsilon^{1/2} \left(\frac{1}{\ell-1}-\epsilon^{1/3}\right)^{r-2}n^{r-2},$$
while by the definition of $X$, the sum of missing multiplicity of $\cM(x)$ is at least
$$\left(k-\binom{\ell}{2}+1\right)\epsilon^{1/2}n^{r-2}>k\frac{r}{2} \epsilon^{1/2} \left(\frac{1}{\ell-1}-\epsilon^{1/3}\right)^{r-2}n^{r-2},$$
when $k>3\binom{\ell}{2}$.
Summing over all $x\in X$, we have
$$\sum_{E \in \mathcal{E}_1 \cup \mathcal{E}_2} m(E) < \sum_{E \in \mathcal{M}} m'(E),$$
and (\ref{eq: missing degree goal}) holds, and we are done.
\hfill
$\square$

\section{Proof of Theorem \ref{thm: small-s}}

In this section, we give the proof of Theorem \ref{thm: small-s}. First, we prove for the lower bound.

\subsection{The lower bound}

When $\ell>r$ and $\binom{\ell}{2}\le s\leq \frac{\ell^2-1}{2}$, let $S_1$ be a set of $\binom{\ell}{2}$ vertices. Then it is easy to check the hypergraph with hyperedges containing exactly one vertex in $S_1$ achieves the lower bound, and is $\{K_{\ell+1}^{(r)+},M_{s+1}\}$-free.
This gives the lower bound of (i) in Theorem \ref{thm: small-s}.

When $\ell>r$ and $\binom{\ell-1}{2}+r\leq s< \binom{\ell}{2}$, let $S_2$ be a set of $s$ vertices. Then it is easy to check the hypergraph with hyperedges containing exactly one vertex in $S_2$ achieves the lower bound, and is $\{K_{\ell+1}^{(r)+},M_{s+1}\}$-free.
This gives the lower bound of (ii) in Theorem \ref{thm: small-s}.
Then, we will prove the lower bound of (iii), (iv) and (v) in Theorem \ref{thm: small-s}.
\begin{construction}\label{cons1}
Suppose $2+\binom{\ell-1}{2}\leq s< r+\binom{\ell-1}{2}$. Let $\cG_1$ be the $r$-graph defined as follows: take a vertex set $A$ of size $s$, a vertex set $B$ of size $n-s$ and a vertex $u\in A$, and add all hyperedges that contain exactly one vertex from $A$ as well as add all hyperedges containing $u$ and at least $s-\binom{\ell-1}{2}$ other vertices from $A$.
\end{construction}

\begin{construction}\label{cons2}
Suppose $2\leq t\leq \ell-2$ and $\ell+1-t+\binom{t}{2}\leq s< \ell-t+\binom{t+1}{2}$. Let $\cG_2(t)$ be the $r$-graph defined as follows: take a vertex set $A$ of size $s$ and a vertex set $B$ of size $n-s$. We partition $A$ into $\ell-t$ parts, where each part has size $\lfloor\frac{s}{\ell-t}\rfloor$ or $\lceil\frac{s}{\ell-t}\rceil$, we denote them as $A_1,\dots,A_{\ell-t}$. We add all hyperedges containing exactly one vertex from $A$, and the additionally, all the hyperedges
formed by $u_i\in A_i$, $u_j\in A_j$ where $i \neq j$, and one $(r-2)$-subset of $B$.
\end{construction}

\begin{construction}\label{cons3}
Suppose $s< \ell$. Let $\cG_3$ be the $r$-graph defined as follows: take a vertex set $A$ of size $s$ and a vertex set $B$ of size $n-s$, and add all hyperedges that contain at least one vertices of $A$.
\end{construction}

Each hyperedge in $\cG_1,\cG_2(t)$ and $\cG_3$  contains at least one vertex in $A$, and $A$ is a set with $s$ vertices, $\cG_1$, $\cG_2(t)$ and $\cG_3$ are $M_{s+1}$-free. Next, we show that $\cG_1$, $\cG_2(t)$ and $\cG_3$ are $K_{\ell+1}^{(r)+}$-free.

\begin{proposition}
$\cG_1$, $\cG_2(t)$ and $\cG_3$ are $K_{\ell+1}^{(r)+}$-free.
\end{proposition}
\begin{proof}
    First, we prove $\cG_1$ is $K_{\ell+1}^{(r)+}$-free.
Suppose $2+\binom{\ell-1}{2}\leq s< r+\binom{\ell-1}{2}$. Assume that $\cG_1$ contains a copy of $K_{\ell+1}^{(r)+}$. Let us consider the core of $K_{\ell+1}^{(r)+}$, which is denoted by $K$, and let $K=\{k_1,k_2,\dots,k_{\ell+1}\}$.
Notice that each hyperedge in $\cG_1$ contains at least one vertex in $A$, we claim that $|K\cap A|\geq 3$. Otherwise, suppose $|K\cap A|=i$, and assume that $K\cap A=\{k_1,\dots,k_i\}$, when $i>0$.
Since each hyperedges in the copy of $K_{\ell+1}^{(r)+}$ containing $\{k_p,k_q\}$, where $i\le p<q\le \ell+1$, contains at least one vertex in $A$, and these vertices are distinct, which implies $|V(K_{\ell+1}^{(r)+})\cap A|\geq \binom{\ell+1-|K\cap A|}{2}+|K\cap A|$. When $|K\cap A|\leq 2$,
$\binom{\ell+1-|K\cap A|}{2}+|K\cap A|> \binom{\ell-1}{2}+r$, a contradiction.

Thus, $|K\cap A|\geq 3$. If $u\in K\cap A$, then we may assume $\{k_1,k_2,k_3\}\subseteq K\cap A$ and $u=k_1$.
And let $E_{12},E_{13},E_{23}$ be the hyperedges in the copy of $K_{\ell+1}^{(r)+}$ containing $\{k_1,k_2\},\{k_1,k_3\},\{k_2,k_3\}$, respectively.
Then, $E_{12},E_{13},E_{23}$ contains at least two vertices in $A$.
But by the construction of $\cG_1$, each hyperedge in $\cG_1$ containing at least two vertices in $A$ must contain the vertex $u=k_1$. However, $E_{23}$ does not contain $k_1$. A contradiction.
If $u\notin K\cap A$, then all of $E_{12},E_{13},E_{23}$ contain $u$, which is a contradiction.
Thus, $\cG_1$ is $K_{\ell+1}^{(r)+}$-free.

Then, we will prove that $\cG_2(t)$ is $K_{\ell+1}^{(r)+}$-free.
In this case, suppose $2\leq t\leq \ell-2$ and $\ell+1-t+\binom{t}{2}\leq s< \ell-t+\binom{t+1}{2}$. Assume that $\cG_2(t)$ contains a copy of $K_{\ell+1}^{(r)+}$. The core of $K_{\ell+1}^{(r)+}$ is denoted by $K$. Since $s< \ell-t+\binom{t+1}{2}$, we have $|K\cap A|\geq \ell+1-t$,
otherwise, similar to the above proof, the number of vertices in $A$ and in the copy of $K_{\ell+1}^{(r)+}$ is at least $\binom{\ell+1-|K\cap A|}{2}+|K\cap A|\ge \binom{t+1}{2}+\ell-t$, a contradiction by the size of $A$.
By the construction, each vertex in $K\cap A$ is contained in a different part of $A$, because there is no hyperedge containing two vertices in the same part.
But since $A$ has only $\ell-t$ parts, and $|K\cap A|\geq \ell+1-t$, a contradiction.
Thus, $\cG_2(t)$ is $K_{\ell+1}^{(r)+}$-free.

Finally, we prove that $\cG_3$ is $K_{\ell+1}^{(r)+}$-free.
Suppose $s< \ell$. Assume that $\cG_3$ contains a copy of $K_{\ell+1}^{(r)+}$. We consider the vertex set $K$ of the core of $K_{\ell+1}^{(r)+}$. Set $|K\cap A|=x$. Clearly, $0\leq x\leq s$. And also with a similar argument, we have $|K\cap A|\geq \ell+1-x$. Then $|V(K_{\ell+1}^{(r)+})\cap A|\geq x+\binom{\ell+1-x}{2}\geq \ell>|A|$, which is a contradiction. Thus, $\cG_3$ is $K_{\ell+1}^{(r)+}$-free.
\end{proof}
The lower bound of (iii), (iv) and (v) in Theorem \ref{thm: small-s} is given by Constructions \ref{cons1}, \ref{cons2} and \ref{cons3}, respectively.

\subsection{Preliminary for the upper bound}


Let $\mathcal{H}$ be an $n$-vertex $r$-graph with the maximum number of hyperedges that is $\{K_\ell^{(r)+}, M_{s+1}^{(r)+}\}$-free. 
Let $V_0 = \{ v \in V(\mathcal{H}) : d_{\mathcal{H}}(v) \geq r(s+1)n^{r-2} \}$ be the set of high-degree vertices. 
In \cite{yang2025hypergraph}, they give an upper bound on the size of $V_0$ by the matching number of $\mathcal{H}$. And they showed that the number of hyperedges contained in $V(\cH)\setminus V_0$ is $O(n^{r-2})$.

\begin{lemma}\label{lem: size of V0}
    We have $|V_0|\leq s$, and the number of hyperedges contained in $V(\cH)\setminus V_0$ is $O(n^{r-2})$.
\end{lemma}
By the lower bound construction, when $s<\frac{\ell^2-1}{2}$, we may assume that the number of hyperedges in $\mathcal{H}$ satisfies
\begin{equation}\label{eq: lower bound H}
	e(\mathcal{H}) \geq \min\left\{s, \binom{\ell}{2}\right\} \binom{n}{r-1} + O(n^{r-2}).
\end{equation}
For a set of vertex $V'$,
let $\mathcal{M}(V')$ denote the set of missing edges that intersect $V'$ in exactly one vertex, defined as $\mathcal{M}(V') = \{ E \in \binom{V(\cH)}{r} : |E \cap V'| = 1 \text{ and } E \notin E(\mathcal{H}) \}$.
We have the following results.

\begin{lemma} \label{lem: find clique with bad edge}
    Let $V'$ be a set of vertices with size at most $s$.
    If the number of hyperedges in $\cH$ is at least $|V'|\binom{n}{r-1}+O(n^{r-2})$, the number of hyperedges contained in $V(\cH)\setminus V'$ is $O(n^{r-2})$, and $|\mathcal{M}(V')| = O(n^{r-2})$. Let $u_1, \dots, u_t$ be $t$ distinct vertices in $V'$, $D \subseteq V' \setminus \{u_1, \dots, u_t\}$ be a $d$-set, and $\ell$ be a positive integer. If $|V'| \geq t + d + \binom{\ell-1}{2}$, then there exists a copy of $K_{\ell-1}^{(r)+}$, denoted by $\mathcal{K}$, such that for each $i \in [t]$, there is a copy of $K_{\ell}^{(r)+}$, denoted by $\mathcal{K}_{u_i}$, satisfying the following:
	\begin{enumerate}
		\item[(i)] $\mathcal{K} \subset \mathcal{K}_{u_i}$ and $u_i \in V(\mathcal{K}_{u_i}) \setminus V(\mathcal{K})$ for each $i \in [t]$;
		\item[(ii)] $V(\mathcal{K}_{u_i}) \cap D = \emptyset$ for all $i \in [t]$;
		\item[(iii)] $V(\mathcal{K}_{u_i}) \cap V(\mathcal{K}_{u_j}) = V(\mathcal{K})$ for all $1 \le i < j \le t$.
	\end{enumerate}
\end{lemma}
\begin{proof}
	Let $\mathcal{F}$ denote the collection of $(r-1)$-sets defined as follows: 
	\[ \mathcal{F} = \{ S \subseteq V(\cH) \setminus V' : |S|=r-1 \text{ and } \{u\} \cup S \in E(\mathcal{H}) \text{ for all } u \in V' \}. \] 
	Let $\mathcal{E} = \{ E \in E(\mathcal{H}) : |E \cap V'| \geq 2 \}$. 
	Since $|V'| \le s$, it is clear that $|\mathcal{E}| = O(n^{r-2})$. 
	Note that each edge in $E(\mathcal{H}) \setminus \mathcal{E}$ contains at most one vertex in $V'$. Combining the lower bound on $e(\mathcal{H})$, we have
	\[ |V'| \binom{n-s}{r-1} - O(n^{r-2}) \leq e(\mathcal{H}) \leq s \cdot |\mathcal{F}| + |\mathcal{E}| + O(n^{r-2}), \]
	which yields $|\mathcal{F}| = \binom{n-s}{r-1} - O(n^{r-2})$.

	A vertex $w \in V(\mathcal{H}) \setminus V'$ is called \textit{good} if for every $i \in [t]$, the number of hyperedges in $E(\mathcal{H}) \setminus \mathcal{E}$ containing both $w$ and $u_i$ is at least $\frac{1}{2} \binom{n-s-1}{r-2}$; otherwise, $w$ is \textit{bad}. Let $B$ and $C$ denote the sets of bad and good vertices in $V(\mathcal{H}) \setminus V'$, respectively. For each bad vertex $w$, there are at least $\frac{1}{3} \binom{n}{r-2}$ missing edges incident to $w$. Since the total number of missing edges $|\mathcal{M}| = O(n^{r-2})$, we have $|B| = O(1)$.
	
	Define $\mathcal{F}' = \{ E \in \mathcal{F} : E \subseteq C\}$. Since $|B| = O(1)$, it follows that $|\mathcal{F}'| = \binom{n-s}{r-1} - O(n^{r-2})$. By the Theorem \ref{Pikhurko}, there exists a copy of $K_{\ell-1}^{(r-1)+}$, denoted by $\mathcal{K}_1$, in $\mathcal{F}'$. By the definition of $\mathcal{F}$ and $s \ge t + d + \binom{\ell-1}{2}$, we can greedily select $\binom{\ell-1}{2}$ distinct vertices from $V' \setminus (D \cup \{u_1, \dots, u_t\})$ to extend $\mathcal{K}_1$ into a copy of $K_{\ell-1}^{(r)+}$, denoted by $\mathcal{K}$. By construction, the core vertices of $\mathcal{K}$, denoted by $\{w_1, \dots, w_{\ell-1}\}$, are all good.
	
	Since each $w_j \in V(\mathcal{K}) \setminus V'$ is a good vertex, it is incident to at least $\frac{1}{2} \binom{n-s-1}{r-2}$ hyperedges containing $u_i$ for each $i \in [t]$. For each $u_i$, we can greedily select $(r-2)$-sets containing $u_i$ and $w_j$, while avoiding the vertices in $V(\cK_1)\cup (V'\setminus \{u_i\})\bigcup_{j<i}V(\cK_j)$.
    This procedure forms the clique $\mathcal{K}_{u_i}$.
     For sufficiently large $n$, the degree condition of good vertices ensures that these newly chosen vertices can be picked to avoid all vertices in $V(\cK_1)\cup (V'\setminus \{u_i\})$ and all vertices previously selected for other $\mathcal{K}_{u_j}$ where $j \neq i$. 
	
	Consequently, for each $i \in [t]$, we obtain a copy of $K_{\ell}^{(r)+}$ such that $V(\mathcal{K}_{u_i}) \cap V(\mathcal{K}_{u_j}) = V(\mathcal{K})$ for all $1 \le i < j \le t$. 
\end{proof}

\subsection{Proof of Theorem \ref{thm: small-s} (i) and (ii).}
We first consider an easier case, where $\ell > r$ and $\binom{\ell-1}{2} + r \le s < \binom{\ell}{2}$.

\noindent
\textbf{Proof of Theorem \ref{thm: small-s} (ii).}
Since every vertex in $V_0$ is contained at most $\binom{n-1}{r-1}$ hyperedges, combined with the lower bound, we have $|V_0|=s$. It also implies that there are no hyperedges contained in $V(\cH)\setminus V_0$. Otherwise, since each vertex in $V_0$ is contained in at least $\binom{n}{r-1}-O(n^{r-2})$ hyperedges, we can greedily find a copy of $M_{s+1}^{(r)+}$.
 By comparing $e(\mathcal{H})$ with the lower bound construction, it follows that $|\mathcal{M}(V_0)| = O(n^{r-2})$.

\begin{claim}\label{claim: one vertex in V0}
	Every hyperedge in $\mathcal{H}$ contains exactly one vertex from $V_0$.
\end{claim}

\begin{proof}
	As established earlier, no hyperedge is contained in $V(\mathcal{H}) \setminus V_0$. Thus, it remains to show that no hyperedge intersects $V_0$ in two or more vertices. 
	Suppose, for the sake of contradiction, that there exists an edge $E \in E(\mathcal{H})$ such that $|E \cap V_0| \ge 2$. Let $u_1, u_2 \in E \cap V_0$ and set $D = E \setminus \{u_1, u_2\}$. 
	Given $s \ge \binom{\ell-1}{2} + r$, we can apply Lemma \ref{lem: find clique with bad edge} for $V_0$ to find two expansion-cliques $\mathcal{K}_{u_1}$ and $\mathcal{K}_{u_2}$, which intersect at $\mathcal{K}$ and avoid $D$. By construction, $\mathcal{K}_{u_1} \cup \mathcal{K}_{u_2} \cup \{E\}$ forms a copy of $K_{\ell+1}^{(r)+}$. This contradiction completes the proof of the claim.
\end{proof}

As established previously, no hyperedge is contained in $V(\mathcal{H}) \setminus V_0$. Combined with Claim \ref{claim: one vertex in V0}, this ensures that every hyperedge in $\mathcal{H}$ contains exactly one vertex in $V_0$.  We have $e(\mathcal{H}) \leq  s \binom{n-s}{r-1}$, and we are done.
$\square$

Then we work on the left case.

\noindent \textbf{Proof of Theorem \ref{thm: small-s} (i).} 
For $\binom{\ell}{2} \le s < \frac{\ell^2-1}{2}$, we define $V_0$ as before. By Lemma \ref{lem: size of V0}, we have $|V_0| \le s$.
Let $\mathcal{H}'$ be the sub-hypergraph obtained by removing from $\mathcal{H}$ all edges that either are contained in $V \setminus V_0$ or intersect $V_0$ in at least two vertices. 
According to Lemma \ref{lem: size of V0} and the fact that $|V_0|$ is constant, the number of such removed hyperedges is $O(n^{r-2})$. 
Consequently, every hyperedge in $\mathcal{H}'$ intersects $V_0$ in exactly one vertex, and its total size satisfies 
$$ 
e(\mathcal{H}') \ge \binom{\ell}{2} \binom{n-s}{r-1} - O(n^{r-2}). 
$$
Following the same argument as in the proof of Theorem \ref{thm: k small rainbow}, we may assume that for every vertex $u \in V(\mathcal{H}) \setminus V_0$, its degree in $\mathcal{H}'$ satisfies 
$$
d_{\mathcal{H}'}(u) \ge \binom{\ell}{2} \binom{n-2}{r-2} - O(n^{r-3}). 
$$
For each vertex $u \in V(\mathcal{H}') \setminus V_0$, define its \textit{heavy} neighbor set in $V_0$ as:
\begin{equation}\label{eq: C(u)}
C(u) = \left\{ v \in V_0 : d_{\mathcal{H}'}(\{u, v\}) \geq (1-\epsilon) \binom{n}{r-2} \right\}, 
\end{equation}
where $\epsilon > 0$ is a constant depending on $\ell$ to be determined later. 
By the minimum degree condition on $\mathcal{H}'$, it follows that $|C(u)| \geq \binom{\ell}{2}$ for every $u \in V(\mathcal{H}') \setminus V_0$.

For each $v_i \in V_0$, let $\mathcal{L}_i = \{ S \subseteq V \setminus V_0 : S \cup \{v_i\} \in E(\mathcal{H}') \}$ be the link graph of $v_i$ in $\mathcal{H}'$. 
If 
\[ \sum_{i=1}^{|V_0|} |\mathcal{L}_i| > \binom{\ell}{2} \binom{n-|V_0|}{r-1}, \]
then by Theorem \ref{thm: k small rainbow} and Corollary \ref{cor: super rainbow}, there exists a super rainbow copy of $K_\ell^{(r-1)+}$ (denoted by $\cK_1$) among the collection of link graphs $\{\mathcal{L}_i\}_{i \in [|V_0|]}$.

Suppose the core vertices of this super rainbow $K_\ell^{(r-1)+}$ are $\{w_1, \dots, w_\ell\}$. 
Since each core vertex $w_i$ satisfies $|C(w_i)| \geq \binom{\ell}{2}$, it follows that at most $|V_0| - \binom{\ell}{2} \leq s - \binom{\ell}{2} < \frac{\ell^2-1}{2} - \binom{\ell}{2} = \frac{\ell-1}{2}$ vertices in $V_0$ are excluded from $C(w_i)$. 

By the Pigeonhole Principle, since $|V_0| \geq \binom{\ell}{2}$, there must exist a vertex $v^* \in \bigcap_{i=1}^\ell C(w_i)$. 
For each $(r-1)$-set $F$ in $\binom{V(\cH')\setminus V_0}{r-1}$, we say $F$ has color $v_i$ for $v_i\in V_0$ if $F$ is a hyperedge in the link graph $\cL_i$.
By the definition of a super rainbow copy, we can choose color for each hyperedge of $\cK_1$ such that $\cK_1$ is rainbow, and each hyperedge is not colored by $v^*$. 
Consequently, by taking $v^*$ as the $(\ell+1)$-th core vertex and using the fact that $v^* \in C(w_i)$ for all $i \in [\ell]$, we can expand the super rainbow $K_\ell^{(r-1)+}$ into a $K_{\ell+1}^{(r)+}$ in $\mathcal{H}'$, which is a contradiction.

Thus, we must have 
\[ \sum_{i=1}^{|V_0|} |\mathcal{L}_i| \leq \binom{\ell}{2} \binom{n-|V_0|}{r-1}. \]
Let $\mathcal{E}$ denote the set of hyperedges in $\mathcal{H}$ that are either contained in $V(\mathcal{H}) \setminus V_0$ or intersect $V_0$ in at least two vertices. 
By Lemma \ref{lem: size of V0} and the fact that $|V_0|$ is constant, we have $|\mathcal{E}| = O(n^{r-2})$. 
Since $E(\mathcal{H}) = E(\mathcal{H}') \cup \mathcal{E}$, it follows that
\[ e(\mathcal{H}) \leq \sum_{i=1}^{|V_0|} |\mathcal{L}_i| + |\mathcal{E}| \leq \binom{\ell}{2} \binom{n}{r-1} + O(n^{r-2}), \]
providing an asymptotically tight upper bound for $e(\mathcal{H})$.

On the other hand, $$\sum_{i\in [k]}|\cL_i|\geq e(\cH)-|\cE|\geq \binom{\ell}{2}\cdot \binom{n}{r-1}-O(n^{r-2}).$$
According to the stability Theorem \ref{thm : stability of rainbow}, we may assume we can delete $O(n^{r-2})$ hyperedges from $\cH'$ such that $\cL_i=\emptyset$ for $i\geq \binom{\ell}{2}+1$, and all the remaining $\binom{\ell}{2}\cdot\binom{n}{r-1}-O(n^{r-2})$ hyperedges contains exactly one vertex in $\{v_1,\dots,v_{\binom{\ell}{2}}\}$. 
Combined with the lower bound of $e(\cH)$,it implies for every $i\leq \binom{\ell}{2}$, $|\cL_i|\geq \binom{n}{r-1}-O(n^{r-2})$.
Let $V'=\{v_1,\dots,v_{\binom{\ell}{2}}\}$.
For every $r$-set $E\subseteq V(\cH)$ with $|E\cap V'|=1$ and $E\notin E(\cH)$, then we collect it in $\cM$ as missing hyperedges. Here $\cM=\cM(V')$ by previous definition.
Then we have $|\cM|=O(n^{r-2})$.

\begin{claim}\label{claim: no hyperedges contains at least two vertices in V'}
    There are no hyperedges contains at least two vertices in $V'$.
\end{claim}
\begin{proof}
    Notice the number of hyperedges contained in $V(\cH)\setminus V'$ is $O(n^{r-2})$.
    By Lemma \ref{lem: find clique with bad edge} on the vertex $V'$, and with a similar argument as Claim \ref{claim: one vertex in V0}, we can prove that if there exists a hyperedge violating the claim, then we can find a copy of $K_{\ell+1}^{(r)+}$.
\end{proof}

For every hyperedge $E\in E(\cH)$, and $E\cap V'=\emptyset$, we collect it as extra hyperedges in set $\cE$.
For every $u\in V(\cH)\setminus V'$ if $V'\subseteq C(u)$, then we collect it in set $A$. Here, $C(u)$ is defined in (\ref{eq: C(u)}).

Then for every $u\in V(\cH)\setminus(A\cup V')$, according to the definition, $u$ is contained in at least $\epsilon\binom{n}{r-1}$ missing hyperedges in $\cM$. Since $|\cM|=O(n^{r-2})$ and $\epsilon=\epsilon(\ell)$ is a constant, we have $|V(\cH)\setminus(A\cup V')|=O(1)$, or equivalently $|A|=n-O(1)$.
\begin{claim}\label{claim: no intersection two with A}
We may assume that there is no $E\in \cE$, with $|E\cap A|\geq 2$.
\end{claim}
\begin{proof}
    In the proof of this claim, we let $\epsilon\leq \frac{1}{\ell+1}$.
    Suppose $w_1,w_2\in A$ are contained in such $E$.
    For $i\in[\binom{\ell}{2}]$ and $j\in[2]$,
    Let $\cD_{i,j}$ denote the $(r-2)$-sets contained in $A$ avoiding $E$ that forms a hyperedge in $\cH$ together with $\{v_i,w_j\}$.
    Then, according to the definition of $A$, we have $|\cD_{i,j}|\geq \frac{2}{3}\binom{n}{r-2}$.
    Then, for integer $i,a\in [\binom{\ell}{2}]$ and $j,b\in [2]$, $|\cD_{i,j}\cap \cD_{a,b}|\geq \frac{1}{3}\binom{n}{r-2}$.

    More Specifically, $|\cD_{1,1}\cap \cD_{2,2}|\geq \frac{1}{3}\binom{n}{r-2}$.
    According to the Tur\'an number of sunflower (Theorem \ref{thm: sunflower}), there exists a $(r-2)$-uniform sunflower with two petals ($P_1,P_2$) contained in $A$ with hyperedges in $\cD_{1,1}\cap \cD_{2,2}$, whose core is denoted by $w_3$.
    And let $E_1,E_2\in E(\cH)$ be the hyperedges containing $P_1$ and $P_2$ respectively; together with $E$, they form a copy of $K_{3}^{(r)+}$.
    We are inductively doing this process.
    Suppose we have found a copy of $K_i^{(r)+}$ (denoted by $\cK_i$) with core vertices $w_1,\dots,w_i\in A$, and $V(\cK_i)\cap V_0=\{v_1,\dots,v_{\binom{i}{2}-1}\}$.
    Then we consider the intersecting of $\cD_{1,\binom{i}{2}},\cD_{2,\binom{i}{2}+1},\cD_{i,\binom{i}{2}+i-1}=\cD_{i,\binom{i+1}{2}-1}$. According to the definition of $C(w_i)$, we have
    $$|\bigcap_{j=1}^i\cD_{j,\binom{i}{2}-1+j}|\geq \frac{1}{\ell+1}\binom{n}{r-2}.$$
    Then, the number of $(r-2)$-sets contained in $A\setminus V(\cK_i)$ and in $\bigcap_{j=1}^i\cD_{j,\binom{i}{2}-1+j}$ is at least $\frac{1}{\ell+1}\binom{n}{r-2}-O(n^{r-3})$.
    Again, according to Theorem \ref{thm: sunflower}, there exists a sunflower with $i$ petals contained in $A$ in  $\bigcap_{j=1}^i\cD_{j,\binom{i}{2}-1+j}$.
    Then, together with $\cK_i$, it forms a copy of $K_{i+1}^{(r)+}$.

    Repeat this process, and we find a copy of $K_{\ell}^{(r)+}$ with core vertices $w_1,\dots,w_\ell\in A$ (denoted by $\cK_\ell$) and $V(\cK_\ell)\cap V_0=\{v_1\dots,v_{\binom{\ell}{2}-1}\}$.
    Since $w_i\in A$, and by the definition, $v_{\binom{\ell}{2}}\in C(w_i)$ for $i\in[\ell]$.
    Thus, $\cK_\ell^{(r)+}$ can be extended to a $K_{\ell+1}^{(r)+}$ by greedily choosing hyperedges containing $w_i$ and $v_{\binom{\ell}{2}}$ that are disjoint from previously chosen vertices.
\end{proof}
The above claim implies, if $A=V(\cH)\setminus V'$, then $|\cE|=0$.
And then, together with Claim \ref{claim: no hyperedges contains at least two vertices in V'}, the number of hyperedges in $\cH$ is at most $\binom{\ell}{2}\cdot\binom{n-\binom{\ell}{2}}{r-1}$, and we are done.

And if $A\neq V(\cH)\setminus V'$, then for every $v\in V(\cH)\setminus(V'\cup A)$, the number of hyperedges contains $v$ is at most $(n-|A|)^{r-2}\cdot |A|=O(n)$.
When $r\geq 4$, it contradicts the assumption that $d_{\cH'}(v)\geq \binom{\ell}{2}\binom{n}{r-2}-O(n^{r-3})$.

We finish the proof by considering the case when $r=3$.
Let $B=V(\cH)\setminus (V'\cup A)$, then $|B|=O(1)$.
\begin{claim}\label{claim: degree of vertex in B}
    For every $w\in B$, if there is a hyperedge contains $w$, another $w'\in B$, and a vertex $u_1\in A$. Then the number of hyperedges in $\cH'$ containing $w$ and exactly one vertex in $V'$ is at most
    $(\ell-1)n+(\binom{\ell}{2}-\ell+1)\cdot \epsilon \ell n$
\end{claim}
\begin{proof}
    Suppose there exists $\ell$ vertices in $V'$ (may assume they are $v_1,\dots,v_\ell$) such that $d_{\cH'}(\{w,v_i\})\geq \epsilon \ell n$ for $i\in[\ell]$.
    And $E=\{w,w',u_1\}$ is the hyperedges with $w'\in B$ and $u_1\in A$.
    Then, according to the definition of $A$, $d_{\cH'}(u_1,v_{\ell+1})\geq (1-\epsilon)n$, and since $d_{\cH'}(\{w,v_1\})\geq \epsilon \ell n$ and $|A|=n-O(1)$, it implies there exists a vertex $u_2\in A$ such that $\{w,u_2,v_1\}\in E(\cH')$ and $\{u_1,u_2,v_{\ell+1}\}\in E(\cH')$.
    Then, we have found a copy of $K_{3}^{(3+)}$ with core vertices $w,u_1,u_2$.
    With a similar process as in Claim \ref{claim: no intersection two with A}, we can find a copy of $K_{\ell+1}^{(3+)}$ with core vertices $w,u_1,\dots,u_\ell$.

    Thus, the number of $v_i\in V'$ such that $d_{\cH'}(\{w,v_i\})\geq \epsilon \ell n$ is at most $\ell-1$.
    This implies the claim holds.
\end{proof}
We may assume every $w\in B$ is contained in a hyperedge described as in Claim \ref{claim: degree of vertex in B}, otherwise, with Claim \ref{claim: no intersection two with A} and the definition of $A$, the degree of $w$ is at most $(\binom{\ell}{2}-\epsilon)n+O(1)<\binom{\ell}{2}n-O(1)$, a contradiction with the minimum degree assumption.
For every vertex $w\in B$, the hyperedges in $\cH$ containing $w$ has three types.
\begin{itemize}
    \item \textbf{Type 1.} The hyperedges containing $w$ and exactly one vertex in $V_0$, which has size at most $(\ell-1)n+(\binom{\ell}{2}-\ell+1)\cdot \epsilon \ell n$ according to Claim \ref{claim: degree of vertex in B}.
    \item \textbf{Type 2.}  The hyperedges containing $w$ and another two vertices in $B$, which is $O(1)$ since $|B|=O(1)$.
    \item\textbf{Type 3.} The hyperedges containing $w$, another $w'\in B$ and a vertex $u\in A$.
\end{itemize}
For every vertex $w\in B$, by the minimum degree assumption, there are at least $\binom{\ell}{2}n-O(1)$ hyperedges containing $w$. According to Claim \ref{claim: degree of vertex in B}, the size of \textbf{Type 3}  hyperedges is at least $$\binom{\ell}{2}n-(\ell-1)n-\left(\binom{\ell}{2}-\ell+1\right)\cdot \epsilon \ell n-O(1)\geq \left(\binom{\ell}{2}-\ell+\frac{1}{2}\right)n.$$
The inequality holds when $\left(\binom{\ell}{2}-\ell+1\right)\cdot \epsilon \ell<\frac{1}{2}$, i.e. $\epsilon<\frac{1}{2\ell\cdot \left(\binom{\ell}{2}-\ell+1\right)}$.
Then we construct an auxiliary graph $H_B$ with vertex set $B$, and every two vertices $w_1,w_2$ are adjacent if there are at least $\frac{n}{3|B|}$ hyperedges containing $w_1,w_2$ and a vertex $u\in A$.
According to the lower bound on Type 3, for every $w\in B$, we claim that $d_{H_B}(w)\geq \binom{\ell}{2}-\ell$.
Otherwise, the size of \textbf{Type 3} hyperedges containing $w$ is at most $$|B|\cdot \frac{n}{3|B|}+\left(\binom{\ell}{2}-\ell-1 \right)n<\left(\binom{\ell}{2}-\ell+\frac{1}{2}\right)n,$$ a contradiction.

Notice that $\binom{\ell}{2}-\ell\geq \ell-1$ when $\ell\geq 5$, thus, there is a matching with size at $\lceil\frac{\ell-1}{2}\rceil$ in $H_B$.
According to the definition of $H_B$, there is a $3$-matching $M_{\lceil\frac{\ell-1}{2}\rceil}^{(3)+}$ in $\cH$ contained in $A\cup B$. Since every vertex in $V_0$ has large degree, there exists a $3$-matching with size at least $\binom{\ell}{2}+\lceil\frac{\ell-1}{2}\rceil\geq \lc\frac{\ell^2-1}{2}\rc\geq s+1$, a contradiction.
When $\ell=4$, $\binom{\ell}{2}-\ell\geq 2$, which implies there is a $3$-matching with size two contained in $A\cup B$. Together with $\binom{\ell}{2}=6$ vertices in $V_0$, there exists a $3$-matching with size $8=\lc\frac{\ell^2-1}{2}\rc\geq s+1$, a contradiction.
Thus, $B$ is empty, which completes the proof.
 $\square $

\subsection{Proof of Theorem \ref{thm: small-s} (iii)-(v)}
\noindent\textbf{Proof of Theorem \ref{thm: small-s} (iii)}
We set $s=\binom{\ell-1}{2}+t$, where $2\leq t\leq r-1$.
By Lemma \ref{lem: size of V0}, $|V_0| \leq s$. Since each vertex in $V_0$ is incident to at most $\binom{n-1}{r-1}$ hyperedges, combined with the lower bound from Construction 4.2, we must have $|V_0|=s$. It also implies that there are no hyperedges contained in $V(\cH)\setminus V_0$ and $|\mathcal{M}(V_0)| = O(n^{r-2})$.
And we suppose $V_0=\{v_1,\dots,v_s\}$.


In this case, $\mathcal{H}$ may contain hyperedges intersecting $V_0$ in at least $t+1$ vertices. For each $i \in \{1, \dots, r\}$, let $\mathcal{F}_i = \{ E \in E(\mathcal{H}) : |E \cap V_0| = i \}$, and define an auxiliary $i$-graph $\cG_i$ on $V_0$ with edge set
$$ E(\cG_i) = \{ e \subseteq V_0 : |e| = i \text{ and } d_{\cF_i}(e) \geq n^{r-i-1} \}. $$

\begin{claim}\label{clm: No K3}
For $t+1\leq i\leq r$,
    the graph $\cG_i$ is $K_3^{(i)+}$-free.
\end{claim}
\begin{proof}
We collect all the vertices $u\in V(\cH)\setminus V_0$ to $A$ if $V_0\subseteq C(u)$, where $C(u)$ is defined in (\ref{eq: C(u)}).
Then each vertex in $V(\cH)\setminus (V_0\cup A)$ lies in $\epsilon n^{r-2}$ missing hyperedges in $\cM(V_0)$, we have 
$|A|=n-O(1)$.

Suppose there is a $K_3^{(i)+}$ in $\cG_i$ with core vertices $\{v_1,v_2,v_3\}$, then this copy intersecting $A$ with $3i-3$ vertices.
Let $\cF$ be the $(r-1)$-sets $S$ contained in $A$ such that $S\cup \{v\}\in E(\cH)$ for every $v\in V_0$.
Then, according to the lower bound of $e(\cH)$, we have $|\cF|=\binom{n}{r-1}+O(n^{r-2})$.
By Theorem \ref{Pikhurko}, there exist a copy of $K_{\ell-2}^{(r-1)+}$ (denoted by $\cK_1$) contained in $A$ with hyperedges in $\cF$. And suppose the core vertices of $\cK_1$ is $\{w_1,\dots,w_{\ell-2}\}$.
Then $|V(\cK_1)\cap V_0|\geq \binom{\ell-2}{2}$.
Since $w_i\in A$, by the definition of $A$, we have $d_\cH(\{w_i,v_j\})\geq (1-\epsilon)\binom{n}{r-2}$ for every $i\in [\ell-2]$ and $j\in [s]$.
Thus, for every $v_j$ where $j\in[3]$ and $w_i$ where $i\in[\ell-2]$, we can greedily choose hyperedges containing $w_i,v_j$ while avoiding all vertices in $V_0\cup V(\cK_1)$ and the vertices contained in previously chosen hyperedges.

It implies there exists a clique $K_{\ell+1}^{(r)+}$ intersecting $V_0$ with $3+3(i-2)+\binom{\ell-2}{2}\leq \binom{\ell-1}{2}+t$ when $\ell\geq 2r+1$. This yields a copy of $K_{\ell+1}^{(r)+}$.
\end{proof}

Note that $|\cF_1| \leq s \binom{n-s}{r-1}$. Furthermore, for $2 \leq i \leq t$, Lemma~\ref{lem: find clique with bad edge} implies that $\cF_i = \emptyset$.
Because if there is a hyperedge $E$ intersecting $V_0$ with $2\le i\le t$ vertices, then we can choose two vertices $v_1,v_2$ in $V_0$ and $E\setminus \{v_1,v_2\}$ be the set $D$ in Lemma \ref{lem: find clique with bad edge}, and find a clique $K_{\ell+1}^{(r)+}$ together with $E$, a contradiction.
Since $\sum_{i=t+2}^r |\cF_i| = O(n^{r-t-2})$, the lower bound implies that $|\cF_{t+1}| \geq \binom{s-1}{t+1}\binom{n-s}{r-t-1}-O(n^{r-t-2})$. By the definition of $\cG_{t+1}$, we have
\[ |\cF_{t+1}| \leq |\cG_{t+1}| \binom{n-s}{r-t-1} +  \binom{s}{t+1} n^{r-t-2}. \]
As $n$ is sufficiently large, it follows that $|\cG_{t+1}| \geq \binom{s-1}{t+1}$. Given that $\cG_{t+1}$ is an $(t+1)$-graph with $s$ vertices that is $K_3^{(t+1)+}$-free. Since $t\leq r-1$, then $s= \binom{\ell-1}{2}+t\geq 2(t+1)$. Mubayi and Verstra{\"e}te \cite{mubayi2005a} proved that
 $\cG_{t+1}$ is the unique extremal $(t+1)$-graph consisting of $\binom{s-1}{t+1}$ hyperedges, all of which share a fixed common vertex $u_1$. 

We further claim that for every $t' \geq t+1$, every hyperedge in intersecting $V_0$ with size $t'$ must contain $u_1$. Suppose, for the sake of contradiction, there exists $E$ with $|E\cap V_0|=t'$ such that $u_1 \notin E$. Let $u_2, u_3 \in E$. Then there exists a copy of $K_3^{(r)+}$ with core vertices $\{u-2,u_3,u_1\}$ containing $2t+r-1$ vertices in $V_0$. 
Similarly as above, when $\ell\geq 2r+1$, there exists a copy of $K_{\ell+1}^{(r)+}$ intersecting $V_0$ with $2t+r-1+\binom{\ell-2}{2}\leq\binom{\ell-1}{2}+t$ vertices. A contradiction.

Consequently, since all hyperedges in $\cF_{i}$ for $i \geq t+1$ contain $u_1$, we have $|\cF_{t'}| \leq \binom{s-1}{t'} \binom{n-s}{r-t'}$. When we complete the proof by summing up all $\cF_{t'}$ for $t'\geq t$.
$\square$





\noindent\textbf{Proof of Theorem \ref{thm: small-s} (iv)}
We set $s=\ell-t+p+\binom{t}{2}$, then our aim is to prove the extremal value is
$$s\binom{n-s}{r-1}+t_2(s,\ell-t)\binom{n-s}{r-2}+O(n^{r-3}).$$

By Lemma \ref{lem: size of V0}, $|V_0| \leq s$. Since each vertex in $V_0$ is incident to at most $\binom{n-1}{r-1}$ hyperedges, combined with the lower bound from Construction 4.2, we must have $|V_0|=s$. It also implies that there are no hyperedges contained in $V(\cH)\setminus V_0$ and $|\mathcal{M}(V_0)| = O(n^{r-2})$.


In this case, $\mathcal{H}$ may contain hyperedges intersecting $V_0$ in exactly two vertices. Let $\mathcal{F}_2 = \{ E \in E(\mathcal{H}) : |E \cap V_0| = 2 \}$, and define an auxiliary graph $\mathcal{G}_2$ on $V_0$ with edge set
\[ E(\mathcal{G}_2) = \{ \{u,v\} \subseteq V_0 : d_{\mathcal{F}_2}(\{u,v\}) \ge n^{r-3} \}. \]

\begin{claim}
	The graph $\cG_2$ is $K_{\ell+1-t}$-free.
\end{claim}
\begin{proof}
	Otherwise, as $s \geq \ell+1-t + \binom{t}{2}$, with a similar process as in the proof of Claim \ref{clm: No K3}, there exists a clique $K_{\ell+1}^{(r)+}$ intersecting $V_0$ with $\binom{t}{2}+\ell+1-t\leq |V_0|$ vertices, a contradiction.
\end{proof}


By the Tur\'an number of $K_{\ell-t+1}$, we have $|\cG_2| = t_2(s, \ell-t)$, thus the number of hyperedges in $\cF_2$ is at most $t_2(s,\ell-t)\binom{n-2}{r-2}$. Since the number of hyperedges contains at most one vertex in $V_0$ is at most $s\binom{n-s}{r-1}$, we are done.
$\square$

\noindent \textbf{Proof of Theorem \ref{thm: small-s} (v).} 
For $s < \ell$, the result follows directly from the Turán number of $M_{s+1}^{(r)+}$. Specifically, for sufficiently large $n$, we have
\[ \ex_r(n, \{K_{\ell+1}^{(r)+}, M_{s+1}^{(r)+}\}) \leq \ex_r(n, M_{s+1}^{(r)+}) = \binom{n}{r} - \binom{n-s}{r}. \]
Since Construction 4.3 is $\{K_{\ell+1}^{(r)+}, M_{s+1}^{(r)+}\}$-free and contains exactly $\binom{n}{r} - \binom{n-s}{r}$ hyperedges, it is the extremal hypergraph, which completes the proof.
$\square$

    \section{Proof of Theorem \ref{large s}}
For $r\geq 2$, the $r$-graph $\cK_{\ell}^r$ is hypergraph with $\ell$ vertices $\{v_1,\dots,v_\ell\}$, and for every $v_iv_j$, there is a hyperedge $E_{i,j}$ containing $v_i,v_j$. 
For different two pairs $v_iv_j$ and $v_av_b$, $E_{i,j}$ and $E_{a,b}$ are not necessarily different.
Mubayi \cite{mubayi2006a} proved the Tur\'an number of $\cK_{\ell+1}^r$.
\begin{theorem}[\cite{mubayi2006a}]\label{thm: turan cKt}
    Let $n\geq 1$ and $\ell\geq r\geq 2$ be integers. Then
$$\ex_r(n,\cK_{\ell+1}^r)=t_r(n,\ell).$$
\end{theorem}
Now let us start with the proof of Theorem \ref{large s} that we restate here for convenience.

\begin{theorem}
    For integers $r\geq3$, $\ell\geq r$, there exists a function $c(r,\ell)$ such that for $s\geq c(r,\ell)$ and sufficiently large $n$, we have
    $$\ex(n,\{K_{\ell+1}^{(r)+},M_{s+1}\})=s\cdot t_{r-1}(n-s,\ell-1).$$
\end{theorem}
\begin{proof}
    Suppose $\cG$ is an $n$-vertex $r$-graph containing no copy of $K_{\ell+1}^{(r)+}$ and no matching of size $s+1$.
   First, we define $V_0$ as above.
   Then, let $\cE$ denote the collection of hyperedges containing at least two vertices in $V_0$ or contained in $V(\cG)\setminus V_0$.
   Then, $|V_0|\leq s$ and $|\cE|=O(n^{r-2})$.

   For a pair of vertices $x,y\in V(\cG)$, we call them {\em sparse} if the number of hyperedges containing both $x$ and $y$ is at most $r\binom{\ell+1}{2}\cdot n^{r-3}$.
   Then, set $\cA$ denote the hyperedges contains at least one sparse pair, we have
   \begin{equation}\label{eq: the size of cA}
       |\cA|\leq \frac{1}{\binom{r}{2}}\cdot \binom{n}{2}\cdot r\binom{\ell+1}{2}\cdot n^{r-3}\leq \binom{\ell+1}{2}n^{r-1}.
   \end{equation}
   Let $\cG'=\cG\setminus (\cE\cup \cA)$. Then
   \begin{equation}\label{eq: lower bound of cG'}
       e(\cG')\geq s\cdot t_{r-1}(n-s,\ell-1)-r\binom{\ell+1}{2}n^{r-1}-O(n^{r-2}).
   \end{equation}
   \begin{claim}\label{claim: no form cliques}
       There is no copy of $\cK_{\ell+1}^{r}$ in $\cG'$.
   \end{claim}
\begin{proof}
    Otherwise, suppose $u_1,\dots,u_{\ell+1}$ is the core vertices of $\cK_{\ell+1}^{(r)+}$, and for every $u_i,u_j$, since all the hyperedges containing $u_i,u_j$ is not in $\cA$, the number of hyperedges containing $u_i,u_j$ is at least $r\binom{\ell+1}{2} n^{r-3}$.
    Then, we can greedily choose hyperedges to form a copy of $K_{\ell+1}^{(r)+}$.
\end{proof}
For every $v\in V_0$, let $\cL_{\cG'}(v)$ denote the $(r-1)$-uniform link graph of $v$, which is the collection of $(r-1)$-sets $S$ with $S\cup \{v\}\in \cG'$.
Then, for every $v\in V_0$, $\cL_{\cG'}(v)$ is $\cK_{\ell}^{r-1}$-free. Otherwise, we can find a copy of $\cK_{\ell+1}^{r}$, which contradicts Claim \ref{claim: no form cliques}.
Thus, according to the Theorem \ref{thm: turan cKt}, we have
$$|\cL_{\cG'}(u)|\leq t_{r-1}(n,\ell-1).$$
We define
$$Y=\{u\in V_0:|\cL_{\cG'}(u)|\geq t_{r-1}(n,\ell-1)-\delta n^{r-1}\}.$$
Where $\delta$ is a constant depending on $r,\ell$, which we will determine later.
Then $$e(\cG')\leq t_{r-1}(n,\ell-1)|Y|+(s-|Y|)(t_{r-1}(n,\ell-1)-\delta n^r).$$
Moreover, since $|V_0|\leq s$, by (\ref{eq: lower bound of cG'}), we have
\begin{equation}
    |Y|\geq s-\frac{r\binom{\ell+1}{2}}{\delta}.
\end{equation}
Then, when $s$ is large enough compared with $\ell$ and $\frac{1}{\delta}$, $|Y|\geq \binom{\ell+1}{2}$.
For every $u\in Y$, and a fixed constant $\epsilon=\epsilon(r,\ell)$, which we will determine later, let $\delta=\delta(\epsilon)$ defined in Lemma \ref{lem: stability Turan expansion}, then $\delta$ is also a constant depending on $r,\ell$. By Lemma \ref{lem: stability Turan expansion}, there is a partition of $V(\cG)\setminus V_0=V_1^u\cup \dots\cup V_{\ell-1}^u$ with the following properties:
Let $\cK^u$ be the complete $(\ell-1)$-partite graph with partition $V_1^u,\dots,V_{\ell-1}^u$, then $\cL_{\cG'}(u)$ is $\epsilon n^{r-1}$-close to $\cK^u$, and moreover $\left|\,|V_i^u|-\frac{n}{\ell-1}\,\right|<\epsilon^{\frac{1}{2}}n$ for all $i\in[\ell-1]$. Let $\cM^u$ denote the hyperedges in $\cK^u$ but not in $\cL_{\cG'}(u)$. Then, $|\cM^u|
\leq \epsilon n^{r-1}$.
Let $U_i^u\subseteq V_i^u$ denote the vertices $w\in V_i^u$ with $d_{\cM^u}(w)\leq \epsilon^{\frac{2}{3}}n^{r-2}$. Then, $|V_i^u\setminus U_i^u|\leq \epsilon^{\frac{1}{3}}n$.

We pick $\binom{\ell+1}{2}$ vertices from $Y$, and denote them as $Y'$.
Then, according to Lemma 3.5 of \cite{yang2025hypergraph}, for every $u,v\in Y'$, and $p,q\in [\ell-1]$, either $|V_p^u\cap V_q^v|\leq 2\epsilon n$, or $|V_p^u\cap V_q^v|\geq (1-2\epsilon n)\frac{n}{\ell-1}$.
Then, we may assume for every $p\in [\ell-1]$, and $u,v\in Y'$,
$|V_p^u\cap V_p^v|\geq (1-2\epsilon n)\frac{n}{\ell-1}$.

Let $V_i=\cap_{u\in Y'}V_i^u$, then $|V_i|\geq (1-2\binom{\ell+1}{2}\epsilon)\frac{n}{\ell-1}$. And let $U_i=\cap_{u\in Y'}U_i^u$, then 
\begin{equation}\label{eq: lower bound of Ui}
|U_i|\geq \left(1-4\binom{\ell+1}{2}\epsilon^{\frac{1}{3}}\right)\frac{n}{\ell-1}.
\end{equation}
Then we construct an auxiliary $(\ell-1)$-partite graph $H_U$ with partition $U_1,\dots,U_{\ell-1}$, and for $w_i\in U_i, w_j\in U_j$, the edge $w_iw_j\in H_U$ if $d_{\cL_{\cG'}(u)}(\{w_i,w_j\})\geq C(r,\ell)n^{r-4}$ (resp. $\{w_iw_j\in \cL_{\cG'}(u)\}$) for every $u\in Y'\in $ when $r\geq 4$ (resp. when $r=3$).
Where $C(r,\ell)=2\binom{\ell+1}{2}r$ defined as in Section 3.

\begin{claim}\label{claim: bddm, find common neighbour}
    For every $w_i\in U_i$ and $j\neq i$, we have $|N_{H_U}(w_i)\cap U_j|\geq \left(1-\binom{\ell+1}{2}\epsilon^{\frac{1}{2}}\right)\frac{n}{\ell-1}$.
\end{claim}
\begin{proof}
    For every $w_i\in U_i$ and every $u\in Y'$, since $d_{\cM^u}(w_i)\leq \epsilon^{\frac{2}{3}}n^{r-2}$, the number of vertices $w_j'\in U_j$ such that $d_{\cL(\cG')(u)}(\{w_i,w_j'\})<C(r,\ell)n^{r-4}$ (resp. $w_iw_j'\notin E(\cL(\cG')(u))$) when $r\ge 4$ (resp. $r=3$) is at most $\leq \frac{\epsilon^{\frac{2}{3}}n^{r-2}}{(n/2\ell)^{r-3}-C(r,\ell)n^{r-4}}<\frac{1}{2}\epsilon^{\frac{1}{2}}n$.
    The inequality holds when $\epsilon$ is small enough and $n$ is large enough.
    Then we have
    $$|N_{H_U}(w_i)\cap U_j|\geq |U_j|-|Y'|\frac{1}{2}\epsilon^{\frac{1}{2}}n\geq (1-\epsilon^{\frac{1}{2}})\frac{n}{\ell-1}-\frac{1}{2}\binom{\ell+1}{2}\epsilon^{\frac{1}{2}}n>\left(1-\binom{\ell+1}{2}\epsilon^{\frac{1}{2}}\right)\frac{n}{\ell-1}.$$
\end{proof}

\begin{claim}\label{claim: bddm no intersection two}
   For every hyperedge $E\in E(\cG)$, we have
    $|E\cap U_i|\leq 1$
    for every $i\in [\ell-1]$.
\end{claim}
\begin{proof}
    Suppose, for a contradiction, that there exists $E\in E(\cG)$ such that  $|E\cap U_1|\geq 2$.
    Then, let $u_0,u_1\in E\cap U_1$, and by Claim \ref{claim: bddm, find common neighbour}, there exists $u_i\in U_i$ for $i\geq 2$ such that $u_iu_j\in H_U$ for all $i<j$ and $(i,j)\neq (0,1)$.
    Moreover, since $u_i \in B_i\subseteq U_i^{v'}$, by the definition of $U_i^{v'}$, there exists $v'\in Y'\setminus E$ such that $d_{\cG_1}(\{v',u_i\})\geq C(r,\ell)n^{r-3}$, otherwise we have $d_{\cM^{v'}}(u_i)\geq \epsilon^{\frac{2}{3}}n^{r-2}$, a contradiction.
    According to the definition of $H_U$, we can greedily choose the hyperedge containing $u_iu_j$ and $u_iv'$, and together with $E$, there exists a copy of $K_{\ell+1}^{(r)+}$, a contradiction.
\end{proof}

\begin{claim}\label{claim: no E intersecting 2}
    There is no hyperedge $E\in E(\cG)$ such that $|E\cap Y'|\geq 2$.
\end{claim}
\begin{proof}
    Suppose $E$ intersecting $Y'$ with at least two vertices, and $v_1,v_2\in Y'$.
    Then, by Claim \ref{claim: bddm, find common neighbour}, there exists a clique $K_{\ell-1}$, with vertices $\{u_1,\dots,u_{\ell-1}\}$ and $u_i\in U_i$, in the graph $H_U$, and $\{u_1,\dots,u_{\ell-1}\}\cap E=\emptyset$.
    By the definition of $H_u$, this copy of $K_{\ell-1}$ can be extended to a copy of $K_{\ell-1}^{(r)+}$ while avoiding the vertices in $E$.
    Because for each vertex $v\in Y'$ and $u_i,u_j$ with $i,j\in [\ell-1]$, we have
    $d_{\cG}(\{u_i,u_j,v\})\geq C(r,\ell)n^{r-4}$ when $r\geq 4$ and $u_iu_j v\in E(\cG)$ when $r=3$.

    And by Claim \ref{claim: bddm no intersection two}, we have $|E\cap U_i|\leq 1$ for every $i\in [\ell-1]$.
    And since $u_i\in U_i^{v_j}$ where $i\in [\ell-1]$ and $j\in [2]$, the number of hyperedges containing $\{u_i,v_j\}$ is $O(n^{r-2})$.
    Thus, we can greedily choose hyperedges containing $\{u_i,v_j\}$ avoiding the previous vertices in the copy of $K_{\ell-1}^{(r)+}$ and the vertices in $E\setminus\{v_1,v_2\}$, which forms a copy of $K_{\ell+1}^{(r)+}$.
\end{proof}

With a similar argument as in the proof of Theorem \ref{thm: k small rainbow} and in (\ref{eq: mini multi}), we may assume that for each vertex $v\in (\cG)$, $d_{\cG}(v)\geq s\cdot t_{r-2}((1-\frac{1}{\ell-1})n,\ell-2)-1$.
By (\ref{eq: lower bound of Ui}), we have $V(\cG)\setminus(\bigcup_{i=1}^{\ell-1}U_i)\leq 4\binom{\ell+1}{2}\epsilon^{\frac{1}{3}}n$.
And $|U_i|\leq \left(\frac{1}{\ell-1}+4\binom{\ell+1}{2}\epsilon^{\frac{1}{3}}\right) n$.

According to Claim \ref{claim: no E intersecting 2}, for each $v\in V_0$ and $u_i\in U_i$, we have
$$ d_{\cL(\cG)(v)}(u_i)\leq \binom{\ell-2}{r-2}\left(\frac{1}{\ell-1}+4\binom{\ell+1}{2}\epsilon^{\frac{1}{3}} \right)^{r-2}n^{r-2}+4\binom{\ell+1}{2}\epsilon^{\frac{1}{3}}n^{r-2}:=\alpha(r,\ell,\epsilon,n).$$
And for a vertex $u\in V(\cG)\setminus V_0$, let $x(u)$ denote the number of vertices $v\in V_0$ such that $d_{\cG_1}(\{u,v\})\geq C(r,\ell)n^{r-3}$.
Then, for each vertex $u\in V(\cG)\setminus V_0$, the number of hyperedges containing $u$ is at most 
$$x(u)\binom{n}{r-2}+(s-x(u))n^{r-3}\geq s\cdot t_{r-2}\left(\left(1-\frac{1}{\ell-1} \right)n,\ell-2\right)-1,$$
Then it implies there exists a constant $b(r,\ell)<\frac{2\ell}{2\ell+1}$ such that $x(u)\geq b(r,\ell)s$.

Then, the number of hyperedges in $\cG$ containing $u_i$ is at most $$\alpha(r,\ell,\epsilon,n)\cdot x(u_i)+(s-x(u_i))C(r,\ell)n^{r-3}\geq s\cdot t_{r-2}((1-\frac{1}{\ell-1})n,\ell-2)-1.$$
The first part counts the hyperedges in $\cL_{\cG}(u_i)$ contained in $\bigcup_{i=1}^{\ell-1}U_i$, and the second part counts the hyperedges in $\cL_{\cG}(u_i)$ containing $u_i$ and at least one vertex in $V(\cG)\setminus \bigcup_{i=1}^{\ell-1}U_i$.
Notice that $t_{r-2}((1-\frac{1}{\ell-1})n,\ell-2)=\binom{\ell-2}{r-2}\left(\frac{1}{\ell-1}\right)^{r-2}n^{r-2}$.
Compare it with $\alpha(r,\ell,\epsilon,n)$, we have $x(u_i)\geq \left(1-\frac{b(r,\ell)}{2}\right)s$ for each $u_i\in U_i$ and $i\in [\ell-1]$ when $\epsilon$ is small enough.

Let $\cG_1=\cG-\cE$, there $\cE$ is the collection of hyperedges either contained in $V(\cG)\setminus V_0$ or containing at least two vertices in $V_0$.
For a $(r-1)$-set $S$ in $V(\cG)\setminus V_0$, if $S$ is contained in at most $\binom{\ell+1}{2}-1$ hypergraphs among $\{\cL_{\cG_1}(v)\}_{v\in V_0}$, then we collect it in $\cF_1$; otherwise, we collect it in $\cF_2$. Then
$$e(\cG_1)\leq |\cF_1|\left(\binom{\ell+1}{2}-1\right)+|\cF_2|\cdot s.$$
For a vertex $u\in V(\cG)\setminus V_0$, let $C_i(u)$ denote the vertices $u_i\in U_i$ such that $d_{\cF_2}(\{u,u_i\})\geq C(r,\ell)n^{r-4}$ when $r\geq 4$ and $uu_i\in \cF_2$ when $r=3$.

For the vertices in $V(\cG)\setminus (V_0\bigcup_{i=1}^{\ell-1}U_i)$, we stepwise add them to $U_i$ by following process.
\begin{itemize}
    \item Let $U_i^0=U_i$ for each $i\in [\ell-1]$, and we processly create $U_i^t$ for $t\geq 0$.
    \item If a vertex $x\in V(\cG)\setminus (V_0\bigcup_{i=1}^{\ell-1}U_i)$ such that there is no hyperedge containing $\{x,u_i\}$ for every $u_i\in U_i^t$,
    and for each $j\neq i$, $|C_j(x)|\geq \frac{2}{3}\frac{n}{\ell-1}$, then let $U_i^{t+1}=U_i^t\cup \{x\}$.
\end{itemize}
When the process ends, suppose each $U_i$ has extended to a set $U_i^{t_i}$, we rename it as $U_i'$.
Then, notice that for each vertex $u_i\in U_i'\setminus U_i$, and each $v\in V_0$, the same bound of $d_{\cL(\cG)(v)}(u_i)$ holds as the vertex in $U_i$, i.e. $d_{\cL(\cG)(v)}(u_i)\leq \alpha(r,\ell,\epsilon,n)$, then $x(u_i)\geq \left(1-\frac{b(r,\ell)}{2}\right)s$.
\begin{claim}\label{claim: bddm no intersection two}
    For every hyperedge $E\in E(\cG)$, we have
     $|E\cap U_i'|\leq 1$
     for every $i\in [\ell-1]$.
 \end{claim}
 \begin{proof}
    The proof is similar to the proof of Claim \ref{claim: bddm no intersection two}, we omit the details.
 \end{proof}

And for a $(r-1)$-set $S\subseteq V(\cG)\setminus V_0$, let $m(S)$ denote the number of $(r-1)$-graphs among $\{\cL_{\cG_1}(v)\}_{v\in V_0}$ containing $S$.
And set $m'(S)=s-m(S)$, denotes the missing multiplicity of $S$.
For each vertex $x\in V(\cG)\setminus \left( V_0\bigcup_{i=1}^{\ell-1}U_i' \right)$, and $i\in [\ell-1]$, let $\cF_2^i(x)$ denote the hyperedges in $\cF_2$ conatining $x$ and intersecting with $U_i'$.
Then, suppose $\cF^{i_x}_2(x)$ reaches the minimum among $\cF_2^i(x)$ for every $i\in [\ell-1]$, then we put $x$ into $X_i'$, and $X'=\bigcup_{i=1}^{\ell-1}X_i'$.
Let $\cK$ be the complete $(\ell-1)$-partite, $(r-1)$-graph with partite $U_1'\cup X_1',\dots,U_{\ell-1}'\cup X_{\ell-1}'$.
By (\ref{eq: lower bound of Ui}), we have 
$$|X'|\leq 4\binom{\ell+1}{2}\epsilon^{\frac{1}{3}}n.$$

And let $\cE_i$ be the $(r-1)$-sets in $\cF_i\setminus \cK$, and for $x\in X'$, let $\cE_i(x)$ be the collection of hyperedges in $\cE_i$ containing $x$.
And $\cM$ be the collection of $(r-1)$-sets in $\cK\setminus \cF_2$, then $\cM(x)$ is the collection of $(r-1)$-sets in $\cM$ containing $x$.

Our aim is to prove for every $x\in X'$, the following claim holds.
\begin{claim}\label{clm: missing degree goal}
    For every $x\in X'$,
    \begin{equation}\label{eq: goal of diff}
   \sum_{S\in \cM(x)}m'(S)- \sum_{S\in \cE_1(x)\cup \cE_2(x)}m(S)\geq \Theta(n^{r-2}).
    \end{equation}
    \end{claim}
\begin{proof}
    For each vertex $x\in X_i'$, according to the definition of $U_i'$, either there is $j\neq i$ such that $C_j(x)< \frac{2}{3}\frac{n}{\ell-1}$, or there is a hyperedge in $\cG$ containing $\{x,u_i\}$ for every $u_i\in U_i'$.

    Suppose there is $j\neq i$ such that $C_j(x)< \frac{2}{3}\frac{n}{\ell-1}$, but no vertex hyperedge in $\cG$ containing $x$ and a vertex $u_i\in U_i'$.
    Then for each vertex $u_j\in U_j\setminus C_j(x)$, there are at least $\binom{\ell-3}{r-3}\left(\frac{1}{2(\ell-1)}\right)^{r-3}n^{r-3}$ hyperedges in $\cM$ by the definition of $C_i(x)$. Then we have 
    $$\sum_{S\in \cM(x)}m'(S)\geq \frac{s-\binom{\ell+1}{2}}{3}\left(\frac{1}{\ell-1}-4\binom{\ell+1}{2}\epsilon^{\frac{1}{3}}\right)n\cdot \binom{\ell-3}{r-3}\left(\frac{1}{2(\ell-1)}\right)^{r-3}n^{r-3}.$$
    On the other hand, the number of hyperedges in $\cE_1(x)\cup \cE_2(x)$ contains at least one other vertex in $X'$. Then we have
    $$\sum_{S\in \cE_1(x)\cup \cE_2(x)}m(S)\leq 4s\binom{\ell+1}{2}\epsilon^{\frac{1}{3}}n^{r-2}.$$
    Then, (\ref{eq: goal of diff}) holds for $x$ when $\epsilon$ is small enough.

    Then we deal with the other case.
   We may assume $x\in X_1'$, and there exists $E\in E(\cG_1)$ containing $x$ and some $u_1\in U_1'$.  
   
   Then, we claim that there is not copy of $K_\ell^{(r-1)+}$ with core vertices $x,u_1$ and another $\ell-2$ core vertices in $\bigcup_{i=2}^{\ell-1}U_i'$ and with hyperedges in $\cF_2$.
   Otherwise, suppose there exists such a copy of $K_{\ell}^{(r-1)+}$ with core vertices $\{x,u-1,u-2,\dots,u_{\ell-1}\}.$ Since $x(x)\geq b(r,\ell)$, where $x(x)$ is the number of vertices $v\in V_0$ such that $d_{\cG_1}(\{x,v\})\geq C(r,\ell)n^{r-3}$. 
   And for each vertex $u_i\in U_i'$, we have $x(u_i)\geq \left(1-\frac{b(r,\ell)}{2}\right)s$.
Thus, there exists $v\in V_0$ such that $d_{\cG_1}(\{w,v\})\geq C(r,\ell)n^{r-3}$ for all $w\in \{x,u_1,u_2,\dots,u_{\ell-1}\}$.

   by the definition of $\cF_2$, we can greedily choose vertex in $V_0\setminus\{v\}$ for each hyperedges of that copy, to form a copy of $K_{\ell}^{(r)+}$.
   Since each core vertex $w$ is contained in at least $C(r,\ell)n^{r-3}$ hyperedges together with $v$, we can greedily choose hyperedges to form a copy of $K_{\ell+1}^{(r)+}$, a contradiction.

   Then, with a similar proof as in Claim \ref{claim: not all in Ci(x)}, we can proof the following holds.
    For every copy of $K_{\ell-1}^{(r-1)+}$ contained in $\cF_1\cup \cF_2$ with core vertices $\{u_1,u_2,\dots,u_{\ell-1}\}$ where $u_i\in U_i'$ for $i\in [\ell-1]$ and $u_1$ is the vertex contained in $E$ together with $x$, there exists at least one vertex $u_i$ such that $u_i\notin C_i(x)$.

   Then, with a similar proof, when $\epsilon$ and $\delta$ are small enough, there exists $i\in [\ell-1]\setminus \{1\}$ with $|U_i'\setminus C_i(x)|\geq \frac{n}{2(\ell-1)}$.
   Then by summing over all choices of $u_i\in U_i'\setminus C_i(x)$, we have
   \begin{equation}\label{eq: lower bound of missing in K}
    \sum_{E\in \cM(x)}m'(E)\geq \frac{s\cdot c(r,\ell)}{2(\ell-1)}n^{r-2}.
   \end{equation}
The constant $c(r,\ell)$ is defined as in Section 3.
And we set $\cA_x$ be the collection of hyperedges in $\cE_1(x)\cup \cE_2(x)$ containing $x$ no other vertex in $X'$, and $\cB_x$ be the hyperedegs in $\cE_1(x)\cup \cE_2(x)$ containing $x$ and at least one other vertex in $X$.
And also with a similar proof as Claim \ref{clm: extral degree total} and the subsequent proof, we have 
$$\sum_{E\in \cA_x\cup \cB_x}m(E)=\sum_{\cE_1(x)\cup \cE_2(x)}m(E)< \frac{5}{6}\sum_{E\in \cM,x\in E}m'(E).$$
Together with (\ref{eq: lower bound of missing in K}), it implies (\ref{eq: goal of diff}) holds.
\end{proof}
Moreover, after deleting the hyperedges in $\cE_1(x)\cup \cE_2(x)$ and add all hyperedges in $\cM(x)$ for each vertex $x\in X'$, all the hyperedges in $\cF_1\cup \cF_2$ forms a sub-hypergraph of a complete $(r-1)$-uniform and $(\ell-1)$-partite hypergraph, then the total size is at most $s\cdot t_{r-1}(n,\ell-1)$.
And the deleting and adding operation adds at least $|X'|\Theta(n^{r-2})$ hyperedges, then we have $s\cdot t_{r-1}(n,\ell-1)-|\cE|\leq e(\cG_1)\leq s\cdot t_{r-1}(n,\ell-1)-|X'|\Theta(n^{r-2})$.
It implies that 
\begin{equation}\label{eq: count with X'}
    e(\cG)-|\cE|=e(\cG_1)\leq s\cdot t_{r-1}(n,\ell-1)-|X'|\Theta(n^{r-2}).
\end{equation}
By the lower bound of $e(\cG)$, we have $|X'|=O_{s,r,\ell}(1)$.

Then for every $v\in V_0$, we have
$|\cL_{\cG_1}(v)|\leq |\cK|+|X'|\cdot n^{r-2}\leq t_{r-1}(n,\ell-1)+O_{s,r,\ell}(n^{r-2})$.
This implies $|V_0|=s$ and $|\cL_{\cG_1}(v)|\geq  t_{r-1}(n,\ell-1)-O(n^{r-2})$.

Then we claim that $\cE=\emptyset$.
Because since each vertex in $V_0$ is contained in $t_{r-1}(n,\ell-1)$ hyperedges, if there exists a hyperedge contained in $V(\cG)\setminus V_0$, then we can greedily choose the hyperedges containing the vertices in $V_0$ to form a copy of $M_{s+1}^{(r)+}$, a contradiction.
And if there exists a hyperedge containing two vertices $v_1,v_2\in V_0$, since each of $v_i$ is contained in at least $t_{r-1}(n,\ell-1)-O(n^{r-2})$ hyperedges, we may assume $v_1,v_2\in Y'$.
Then it is a contradiction with Claim \ref{claim: no E intersecting 2}.

Then, (\ref{eq: count with X'}) implies that $|X'|=0$, and $e(\cG)\leq s\cdot t_{r-1}(n,\ell-1)$, which completes the proof.
\end{proof}

\section{Concluding remarks}
In this paper, we show that the rainbow hyper-Tur\'an problem is closely related to the Tur\'an problem for expansions of graphs.
In particular, this connection is most apparent when the extremal hypergraph has the property that almost every hyperedge intersects a fixed vertex set of constant size.

In the graph case, the rainbow Tur\'an number of $K_\ell$ has two different extremal values, depending on the value of $k$, as shown in Theorem \ref{thmL sudakov rainbow}.
For the $r$-graph case, when $k\in [\frac{\ell^2-1}{2},\,k_0(r,\ell))$, the value of $\ex^{\sum}(n,k,K_\ell^{(r)+})$ has not been determined.
Motivated by the behavior of the rainbow Tur\'an number of $K_\ell$ in the graph case, we have the following conjecture.
\begin{conjecture}
    For integers $\ell\geq r\geq 3$, there is a constant $k_1=k_1(r,\ell)\geq \frac{\ell^2-1}{2}$ such that for sufficiently large $n$, we have when $k<k_1$,
    $$\ex^{\sum}_r(n,k,K_\ell^{(r)+})= \min\left\{k,\left(\binom{\ell}{2}-1\right)\right\}\binom{n}{r},$$
    and when $k\geq k_1$,
    $$\ex^{\sum}_r(n,k,K_\ell^{(r)+})=k\cdot t_r(n,\ell-1).$$
\end{conjecture}
We also refine Conjecture \ref{conj:4.1} as follows.
\begin{conjecture}
    For integers $\ell\geq r\geq 3$, there is a constant $s_1=s_1(r,\ell)\geq \frac{\ell^2}{2}$ such that for sufficiently large $n$, we have when $\binom{\ell}{2}\le s<s_1$,
    $$\ex_r(n,\{K_{\ell+1}^{(r)+},M_{s+1}^{(r)+}\})=\binom{\ell}{2}\cdot \binom{n-\binom{\ell}{2}}{r-1},$$
    and when $s\geq s_1$,
    $$\ex_r(n,\{K_{\ell+1}^{(r)+},M_{s+1}^{(r)+}\})=s\cdot t_r(n,\ell-1).$$
\end{conjecture}

\section*{Acknowledgments}
The research of Zhao is supported by the China Scholarship Council (No. 202506210250) and
the National Natural Science Foundation of China (Grant 12571372).

The research of Wang is supported by the National Nature Science Foundation of China (grant numbers 12331012).

The research of Zhou is supported by the National Natural Science Foundation of China (Nos. 12271337 and 12371347).


\begin{thebibliography}{10}
    \expandafter\ifx\csname urlstyle\endcsname\relax
      \providecommand{\doi}[1]{doi:\discretionary{}{}{}#1}\else
      \providecommand{\doi}{doi:\discretionary{}{}{}\begingroup
      \urlstyle{rm}\Url}\fi
    
    \bibitem{alon2024turan}
    N.~Alon and P.~Frankl.
    \newblock Tur{\'a}n graphs with bounded matching number.
    \newblock \emph{Journal of Combinatorial Theory, Series B}, 165:223--229, 2024.
    
    \bibitem{bollob1976sets}
    B.~Bollob{\'a}s, D.~Daykin, and P.~Erd{\"o}s.
    \newblock Sets of independent edges of a hypergraph.
    \newblock \emph{The Quarterly Journal of Mathematics}, 27(1):25--32, 1976.
    
    \bibitem{bradavc2023turan}
    D.~Brada{\v{c}}, M.~Buci{\'c}, and B.~Sudakov.
    \newblock Tur{\'a}n numbers of sunflowers.
    \newblock \emph{Proceedings of the American Mathematical Society},
      151(03):961--975, 2023.
    
    \bibitem{chakraborti2024rainbowextremal}
    D.~Chakraborti, J.~Kim, H.~Lee, H.~Liu, and J.~Seo.
    \newblock On a rainbow extremal problem for color-critical graphs.
    \newblock \emph{Random Structures \textnormal{\&} Algorithms}, 64(2):460--489, 2024.
    \newblock \doi{10.1002/rsa.21189}.
    \newblock Also available as arXiv:2204.02575.
    
    \bibitem{erdos1965problem}
    P.~Erdos.
    \newblock A problem on independent $r$-tuples.
    \newblock \emph{Ann. Univ. Sci. Budapest. E{\"o}tv{\"o}s Sect. Math},
      8(93-95):2, 1965.
    
    \bibitem{erdHos1971topics}
    P.~Erd{\H{o}}s.
    \newblock Topics in combinatorial analysis.
    \newblock In \emph{Proceedings of the Second Louisiana Conference on
      Combinatorics, Graph Theory and Computing}, pages 2--20. 1971.
    
    \bibitem{gallai1959maximal}
    P.~Erd{\"o}s and T.~Gallai.
    \newblock On maximal paths and circuits of graphs.
    \newblock \emph{Acta Math. Acad. Sci. Hungar}, 10:337--356, 1959.
    
    \bibitem{erdos1966limit}
    P.~Erdos and M.~Simonovits.
    \newblock A limit theorem in graph theory.
    \newblock \emph{Studia Sci. Math. Hungar}, 1(51-57):51, 1966.
    
    \bibitem{erdos1946structure}
    P.~Erd{\"o}s and A.~Stone.
    \newblock On the structure of linear graphs.
    \newblock \emph{Bulletin of the American Mathematical Society},
      52(12):1087--1091, 1946.
    
    \bibitem{frankl2013improved}
    P.~Frankl.
    \newblock Improved bounds for {Erd{\H{o}}s}' matching conjecture.
    \newblock \emph{Journal of Combinatorial Theory, Series A}, 120(5):1068--1072,
      2013.
    
    \bibitem{frankl2017maximum}
    P.~Frankl.
    \newblock On the maximum number of edges in a hypergraph with given matching
      number.
    \newblock \emph{Discrete Applied Mathematics}, 216:562--581, 2017.
    
    \bibitem{frankl2017proof}
    P.~Frankl.
    \newblock Proof of the {Erd{\H{o}}s} matching conjecture in a new range.
    \newblock \emph{Israel Journal of Mathematics}, 222(1):421--430, 2017.
    
    \bibitem{frankl2022erdHos}
    P.~Frankl and A.~Kupavskii.
    \newblock The {Erd{\H{o}}s} matching conjecture and concentration inequalities.
    \newblock \emph{Journal of Combinatorial Theory, Series B}, 157:366--400, 2022.
    
    \bibitem{gerbner2024turan}
    D.~Gerbner.
    \newblock On {Tur{\'a}n} problems with bounded matching number.
    \newblock \emph{Journal of Graph Theory}, 106(1):23--29, 2024.
    
    \bibitem{gerbner2025matchingany}
    D.~Gerbner and S.~Miao.
    \newblock Rainbow {Tur{\'a}n} problems for a matching and any other graph.
    \newblock \emph{arXiv preprint arXiv:2505.14386}, 2025.
    
    \bibitem{gerbner2025hypergraph}
    D.~Gerbner, C.~Tompkins, and J.~Zhou.
    \newblock On hypergraph {Tur{\'a}n} problems with bounded matching number.
    \newblock \emph{European Journal of Combinatorics}, 127:104155, 2025.
    
    \bibitem{huang2012size}
    H.~Huang, P.-S. Loh, and B.~Sudakov.
    \newblock The size of a hypergraph and its matching number.
    \newblock \emph{Combinatorics, Probability and Computing}, 21(3):442--450,
      2012.
    
    \bibitem{keevash2007rainbowturan}
    P.~Keevash, D.~Mubayi, B.~Sudakov, and J.~Verstra{\"e}te.
    \newblock Rainbow {Tur{\'a}n} problems.
    \newblock \emph{Combinatorics, Probability and Computing}, 16(1):109--126,
      2007.
    
    \bibitem{keevash2004multicolour}
    P.~Keevash, M.~Saks, B.~Sudakov, and J.~Verstra{\"e}te.
    \newblock Multicolour {Tur{\'a}n} problems.
    \newblock \emph{Advances in Applied Mathematics}, 33(2):238--262, 2004.
    
    \bibitem{kolupaev2023erdHos}
    D.~Kolupaev and A.~Kupavskii.
    \newblock {Erd{\H{o}}s} matching conjecture for almost perfect matchings.
    \newblock \emph{Discrete Mathematics}, 346(4):113304, 2023.
    
    \bibitem{kupavskii2025complete}
    A.~Kupavskii and G.~Sokolov.
    \newblock A complete solution of the {Erd{\H{o}}s}-{Kleitman} matching problem
      for $n\le 3s$.
    \newblock \emph{arXiv preprint arXiv:2511.21628}, 2025.
    
    \bibitem{li2025multicolorturan}
    X.~Li, J.~Ma, and Z.~Zheng.
    \newblock On the multicolor {Tur{\'a}n} conjecture for color-critical graphs.
    \newblock \emph{Canadian Journal of Mathematics}, 2025.
    \newblock Advance online publication. Also available as arXiv:2407.14905.
    
    \bibitem{luczak2014erdHos}
    T.~{\L}uczak and K.~Mieczkowska.
    \newblock On {Erd{\H{o}}s}' extremal problem on matchings in hypergraphs.
    \newblock \emph{Journal of Combinatorial Theory, Series A}, 124:178--194, 2014.
    
    \bibitem{mubayi2006a}
    D.~Mubayi.
    \newblock A hypergraph extension of {Tur{\'a}n}'s theorem.
    \newblock \emph{Journal of Combinatorial Theory, Series B}, 96(1):122--134,
      2006.
    
    \bibitem{mubayi2005a}
    D.~Mubayi and J.~Verstra{\"e}te.
    \newblock Proof of a conjecture of {Erd{\H{o}}s} on triangles in set-systems.
    \newblock \emph{Combinatorica}, 25(5):599--614, 2005.
    
    \bibitem{mubayi2016a}
    D.~Mubayi and J.~Verstra{\"e}te.
    \newblock A survey of {Tur{\'a}n} problems for expansions.
    \newblock In \emph{Recent trends in combinatorics}, pages 117--143. Springer,
      2016.
    
    \bibitem{pikhurko2013exact}
    O.~Pikhurko.
    \newblock Exact computation of the hypergraph {Tur{\'a}n} function for expanded
      complete 2-graphs.
    \newblock \emph{Journal of Combinatorial Theory, Series B}, 103(2):220--225,
      2013.
    
    \bibitem{turan1941egy}
    P.~Tur{\'a}n.
    \newblock Egy gr{\'a}felm{\'e}leti sz{\'e}lso{\'e}rt{\'e}kfeladatr{\'o}l.
    \newblock \emph{Mat. Fiz. Lapok}, 48(3):436, 1941.
    
    \bibitem{yang2025hypergraph}
    C.~Yang, J.~Zeng, and X.-D. Zhang.
    \newblock A hypergraph analogue of {Alon-Frankl} theorem.
    \newblock \emph{arXiv preprint arXiv:2511.21096}, 2025.

    \bibitem{ZhY}
    J. Zhou and X. Yuan. Linear Tur\'{a}n problems with bounded matching number in hypergraphs. \emph{Discrete Mathematics} 349 (2026) 114772. 
    
    \end{thebibliography}
\end{document}